\documentclass[a4paper,11pt]{article}
\usepackage{amsfonts,amssymb,amsmath, dsfont,relsize,accents}
\usepackage{bm}
\usepackage{mathrsfs}  % curly fonts
\usepackage{color}
\usepackage{upgreek}
\usepackage[english]{babel}
\usepackage{graphicx}
\usepackage{microtype}
\usepackage[colorlinks=true, pdfstartview=FitV, linkcolor=blue, citecolor=blue, urlcolor=blue,pagebackref=false]{hyperref}
\usepackage{graphicx}
\usepackage{epstopdf}

\definecolor{labelkey}{gray}{.8}
\definecolor{refkey}{gray}{.8}

\definecolor{darkred}{rgb}{0.9,0.1,0.1}

 \makeatletter
 \@addtoreset{equation}{section}
 \makeatother

 \makeatletter
 \@addtoreset{enunciato}{section}
 \makeatother

 \newcounter{enunciato}[section]

 \newtheorem{ittheorem}{Theorem}
 \newtheorem{itlemma}{Lemma}
 \newtheorem{itproposition}{Proposition}
 \newtheorem{itcorollary}{Corollary}
 \newtheorem{itdefinition}{Definition}
 \newtheorem{itremark}{Remark}
 \newtheorem{itclaim}{Claim}
 \newtheorem{itfact}{Fact}
 \newtheorem{itconjecture}{Conjecture}

 \newenvironment{theorem}{\addtocounter{enunciato}{1}
 \begin{ittheorem}}{\end{ittheorem}}

 \newenvironment{lemma}{\addtocounter{enunciato}{1}
 \begin{itlemma}}{\end{itlemma}}

 \newenvironment{proposition}{\addtocounter{enunciato}{1}
 \begin{itproposition}}{\end{itproposition}}

 \newenvironment{corollary}{\addtocounter{enunciato}{1}
 \begin{itcorollary}}{\end{itcorollary}}

 \newenvironment{definition}{\addtocounter{enunciato}{1}
 \begin{itdefinition}}{\end{itdefinition}}

 \newenvironment{remark}{\addtocounter{enunciato}{1}
 \begin{itremark}}{\end{itremark}}

 \newenvironment{claim}{\addtocounter{enunciato}{1}
 \begin{itclaim}}{\end{itclaim}}

 \newenvironment{fact}{\addtocounter{enunciato}{1}
 \begin{itfact}}{\end{itfact}}

 \newenvironment{conjecture}{\addtocounter{enunciato}{1}
 \begin{itconjecture}}{\end{itconjecture}}

 \newcommand{\be}[1]{\begin{equation}\label{#1}}
 \newcommand{\ee}{\end{equation}}

 \newcommand{\bl}[1]{\begin{lemma}\label{#1}}
 \newcommand{\el}{\end{lemma}}

 \newcommand{\br}[1]{\begin{remark}\label{#1}}
 \newcommand{\er}{\end{remark}}

 \newcommand{\bt}[1]{\begin{theorem}\label{#1}}
 \newcommand{\et}{\end{theorem}}

 \newcommand{\bd}[1]{\begin{definition}\label{#1}}
 \newcommand{\ed}{\end{definition}}

 \newcommand{\bcl}[1]{\begin{claim}\label{#1}}
 \newcommand{\ecl}{\end{claim}}

 \newcommand{\bfact}[1]{\begin{fact}\label{#1}}
 \newcommand{\efact}{\end{fact}}

 \newcommand{\bp}[1]{\begin{proposition}\label{#1}}
 \newcommand{\ep}{\end{proposition}}

 \newcommand{\bc}[1]{\begin{corollary}\label{#1}}
 \newcommand{\ec}{\end{corollary}}

 \newcommand{\bcj}[1]{\begin{conjecture}\label{#1}}
 \newcommand{\ecj}{\end{conjecture}}

 \newcommand{\bpr}{\begin{proof}}
 \newcommand{\epr}{\end{proof}}

 \newcommand{\bprlem}[1]{\begin{proofof}{\it Lemma \ref{#1}}.\,\,}
 \newcommand{\eprlem}{\end{proofof}}

 \newcommand{\bprthm}[1]{\begin{proofof}{\it Theorem \ref{#1}}.\,\,}
 \newcommand{\eprthm}{\end{proofof}}

 \newcommand{\bprprop}[1]{\begin{proofof}{\it Proposition \ref{#1}}.\,\,}
 \newcommand{\eprprop}{\end{proofof}}

 \newcommand{\bi}{\begin{itemize}}
 \newcommand{\ei}{\end{itemize}}

 \newcommand{\ben}{\begin{enumerate}}
 \newcommand{\een}{\end{enumerate}}

 \newenvironment{proof}{\noindent {\em Proof}.\,\,}{\hspace*{\fill}$\halmos$\medskip}
 \newenvironment{proofof}{\noindent {\em Proof of\,\,}}{\hspace*{\fill}$\halmos$\medskip}
 \newcommand{\halmos}{\rule{1ex}{1.4ex}}

 \newcommand{\one}{{\mathchoice {1\mskip-4mu\mathrm l}
         {1\mskip-4mu\mathrm l}
         {1\mskip-4.5mu\mathrm l}
         {1\mskip-5mu\mathrm l}}}

\def \E {{\mathbb E}}

\def \N {{\mathbb N}}
\def \P {{\mathbb P}}

\def \R {{\mathbb R}}
\def \Z {{\mathbb Z}}

\def \lra \leftrightarrow

\def \ra {\rightarrow}
\def \ba {\begin{array}}
\def \ea {\end{array}}

\def \lra {\longrightarrow}

\def \lra {{\leftrightarrow}}

\def \subset {\subseteq}

\def\one{\rlap{\mbox{\small\rm 1}}\kern.15em 1}

\newlength{\dhatheight}
\newcommand{\doublehat}[1]{%
    \settoheight{\dhatheight}{\ensuremath{\hat{#1}}}%
    \addtolength{\dhatheight}{-0.35ex}%
    \hat{\vphantom{\rule{1pt}{\dhatheight}}%
    \smash{\hat{#1}}}}

\begin{document}
\title{Functional Central Limit Theorem for the Interface of the Multitype Contact Process}

\author{Thomas Mountford\textsuperscript{1}, Daniel Valesin\textsuperscript{2}}
\footnotetext[1]{Ecole Polytechnique F\'ed\'erale de Lausanne. \url{thomas.mountford@epfl.ch}}
\footnotetext[2]{Johann Bernoulli Institute, University of Groningen. \url{d.rodrigues.valesin@rug.nl}}
\date{September 14, 2015}
\maketitle

\begin{abstract}
We study the interface of the multitype contact process on $\Z$. In this process, each site of $\Z$ is either empty or occupied by an individual of one of two species. Each individual dies with rate 1 and attempts to give birth with rate $2 R \lambda$; the position for the possible new individual is chosen uniformly at random within distance $R$ of the parent, and the birth is suppressed if this position is already occupied. We consider the process started from the configuration in which all sites to the left of the origin are occupied by one of the species and all sites to the right of the origin by the other species, and study the evolution of the region of interface between the two species. We prove that, under diffusive scaling, the position of the interface converges to Brownian motion. 
\end{abstract}

{\bf\large{}}\bigskip

%%%%%%%%%%%%%% Introduction %%%%%%%%%%%%%%%%%%%%%%%%%%%%%%%%%%%%%%%%%%%%%%%%%%%%%%%%%%%%%%%%%%%%%%%%%%%%%%%%%%%%%%%%%%%%%%

\section{Introduction}

The multitype contact process is a stochastic process that can be seen as a model for the evolution of different biological species competing for the occupation of space. It was introduced by Neuhauser in \cite{neuhauser} as a modification of Harris' (single-type) contact process (\cite{harris}). 

Let us give the definition of the multitype contact process $(\xi_t)_{t \ge 0}$ on $\Z^d$ with (at most) two types. We will need the parameters: $R_1$, $R_2 \in \mathbb{N}$ and $\delta_1$, $\delta_2$, $\lambda_1$, $\lambda_2>0$. $(\xi_t)_{t \ge 0}$ is then the Markov process with state space $\{0,1,2\}^{\Z^d}$ and generator given by $\mathcal{L} =  \mathcal{L}_1 + \mathcal{L}_2$, with
\begin{equation}\label{eq:generator}(\mathcal{L}_i f)(\xi) = \sum_{\substack{x \in \Z^d:\\\xi(x) = i}} \delta_i \cdot [f(\xi^{0\to x}) - f(\xi)] + \sum_{\substack{x \in \Z^d:\\\xi(x) = 0}}\;\;\sum_{\substack{y \in \Z^d:\\ |x-y|\leq R_i,\\\xi(y) = i}} \lambda_i \cdot [f(\xi^{i\to x}) - f(\xi)],\quad i = 1, 2,\end{equation}
where $f: \{0,1,2\}^{\Z^d} \to \mathbb{R}$ is a function that depends only on finitely many coordinates, $|\cdot|$ is the $\ell^\infty$ norm and
\begin{equation} \label{eq:i_to_x}\xi^{i \to x}(y) = \left\{ \begin{array}{ll}\xi(y)&\text{if } y \neq x;\\i&\text{if } y = x, \end{array} \right.\qquad i = 0, 1, 2.\end{equation}

 We will adopt throughout the paper the following terminology: vertices are called \textit{sites}, sites in state 0, 1 and 2 are respectively said to be \textit{empty} or to have a type 1 or type 2 \textit{occupant} (or \textit{individual}), and elements of $\{0,1,2\}^{\Z^d}$ are called \textit{configurations}. Additionally, $\delta_1, \delta_2$ are called \textit{death rates}, $R_1, R_2$ are \textit{ranges} and $\lambda_1,\lambda_2$ are \textit{birth rates} (or sometimes \textit{infection rates}).

Let us now explain the dynamics in words. Two kinds of transitions can occur. First, an individual of type $i$ dies with rate $\delta_i$, leaving its site empty. Second, given a pair of sites $x, y$ with $|x-y| \leq R$, $\xi(x) = i$ (with $i = 1$ or $2$) and $\xi(y) = 0$, the occupant of $x$ gives birth at $y$ with rate $\lambda_i$, so that a new individual of type $i$ is placed at $y$. Note that, under these rules, births only occur at empty sites, so that the state of a site can never change directly from 1 to 2 or from 2 to 1.

In case only one type (say, type 1) is present, this reduces to the contact process introduced by Harris in \cite{harris}, to be denoted here by $(\zeta_t)_{t \geq 0}$ in order to distinguish it from the multitype version. We refer the reader to \cite{lig99} for an exposition of the contact process and the statements about it that we will gather in this Introduction and in Section \ref{s:back}.

Let $(\zeta^{\{0\}}_t)_{t\geq 0}$ be the (one-type) contact process with rates $\delta_1 = \delta = 1$, $\lambda_1 = \lambda > 0$, $R_1 = R \in \N$ and the initial configuration in which only the origin is occupied. 
Denote by $\underline{0}$ the configuration in which every vertex is empty, and note that this is a trap state for the dynamics. $\;$ There exists $\lambda_c = \lambda_c(\Z, R)$ (depending on the dimension $d$ and the range $R$) such that
\begin{equation}\P\left[\text{there exists } t > 0\text{ such that }\zeta^{\{0\}}_t = \underline{0}\right] = 1 \quad \text{ if and only if } \lambda \leq \lambda_c.\label{eq:phase_transition}\end{equation}
This \textit{phase transition} is the most fundamental property of the contact process. The process is called \textit{subcritical}, \textit{critical} and \textit{supercritical} respectively in the cases $\lambda < \lambda_c$, $\lambda = \lambda_c$ and $\lambda > \lambda_c$.

In this paper, we will consider the multitype contact process $(\xi_t)$ on $\Z$ with parameters
\begin{equation} \delta_1 = \delta_2 = 1,\qquad R_1 = R_2 = R,\qquad \lambda_1 = \lambda_2 = \lambda > \lambda_c(\mathbb{Z},R).\label{eq:parameters}\end{equation}
We emphasize that the quantity $\lambda_c(\Z,R)$ that appears here is the one associated to the \textit{one-type process}, as in \eqref{eq:phase_transition}. We will be particularly interested in the `heaviside' initial configuration,
\begin{equation}\label{eq:def_heaviside}\xi^h_0(x) = \begin{cases} 1&\text{if } x \leq 0;\\2&\text{if } x > 0.\end{cases}\end{equation}
We will denote by $(\xi^h_t)_{t \geq 0}$ the process with rates \eqref{eq:parameters} and initial configuration $\xi^h_0$. We let
\begin{equation}
\label{eq:def_edges}
r_t = \sup\{x:\xi^h_t(x) = 1\},\qquad \ell_t = \inf\{x: \xi^h_t(x) = 2\},\qquad i_t = (r_t + \ell_t)/2.
\end{equation}
The interval delimited by $r_t$ and $\ell_t$ is called the \textit{interface} at time $t$, and $i_t$ is the \textit{position} of the interface at time $t$. The choice of the middle point of the interval as the position of the interface is somewhat arbitrary and will not matter for all the results obtained in this paper.

In case $R = 1$, it follows readily from inspecting the generator in \eqref{eq:generator} that $r_t < \ell_t$ for all $t$. If $R > 1$, both $r_t < \ell_t$ and $r_t > \ell_t$ are possible (in the latter case we say that we have a \textit{positive} interface, and in the previous case, a \textit{negative} interface). In \cite{valesin}, it is shown that the process $(|r_t - \ell_t|)_{t \geq 0}$, which describes the evolution of the \textit{size} of the interface, is stochastically tight:
\begin{theorem}\cite{valesin}
\label{thm:size_inter_tight} If $R \in \N$ and $\lambda > \lambda_c(\Z,R)$, then 
\begin{equation}\label{eq:interface_tightness}\text{for any } \varepsilon > 0 \text{ there exists } L > 0 \text{ such that }\P\left[|r_t - \ell_t| > L\right] < \varepsilon \text{ for all } t \geq 0.\end{equation}
\end{theorem}
In the present paper, we will continue the study of the interface, but we will focus on its position rather than its size. Our main result is
\begin{theorem}
If $R \in \N$ and $\lambda > \lambda_c(\Z,R)$, then there exists $\sigma > 0$ such that 
$$\left(t^{-1/2} \cdot i_{st}\right)_{s \geq 0} \;\overset{t\to \infty}{\underset{\text{(dist.)}}\longrightarrow}\; (B_s)_{s \geq 0},$$
where $(B_s)_{s \geq 0}$ denotes Brownian motion with diffusive constant $\sigma$, and convergence holds in the space $D = D[0,\infty)$ of c\`adl\`ag trajectories with the Skorohod topology.
\end{theorem}

Our proof of this result follows the usual two steps: verifying convergence of finite-dimensional distributions and tightness of trajectories in $D$ (see Section 16 of \cite{bill}). We thus prove the following propositions, both applicable to the case $R \in \N$ and $\lambda > \lambda_c(\Z,R)$:
\begin{proposition}There exists $\sigma > 0$ such that, for any $0 = a_0 < a_1 < \cdots < a_k$ we have
$$t^{-1/2}\cdot \left(i_{a_1\cdot t},\;i_{a_2\cdot t}-i_{a_1 \cdot t},\ldots, i_{a_k \cdot t } - i_{a_{k-1} \cdot t}\right) \;\overset{t\to \infty}{\underset{\text{(dist.)}}\longrightarrow}\; (N_1,\ldots, N_k),$$
where $N_1,\ldots,N_k$ are independent and $N_j\sim \mathcal{N}(0,\sigma^2(a_j - a_{j-1}))$.
\label{prop:fdd} 
\end{proposition}

\begin{proposition}\label{prop:tight}
For any $\varepsilon > 0$ there exists a compact set $K \subset D$ such that 
$$\liminf_{t \to \infty} \P\left[ (t^{-1/2}\cdot i_{st})_{s \geq 0} \in K \right] > 1-\varepsilon.$$
\end{proposition}

In proving these propositions, we will establish a result of independent interest, which we call \textit{interface regeneration}. We will explain it here only informally; the precise result depends on a few definitions and is given in Theorem \ref{thm:interface_regeneration}. Given $s > 0$, consider the configuration $\xi^h_s$ and assume the interface position $i_s = x$. Suppose we define a new configuration $\bar \xi$ by putting 1's in all sites to the left of $i_s$ and 2's to the right of $i_s$. We then show that it is possible to construct, in the same probability space as that of $(\xi_t)_{t \geq 0}$, a multitype contact process started from time $s$, $(\xi'_t)_{t \geq s}$, such that $\xi'_s  = \bar \xi$ and moreover, \textit{the interface positions for $(\xi_t)$ and for $(\xi'_t)$ are never too far from each other}. Since the evolution of the interface of $\xi'$ has the same distribution as that of the original process (except for a space-time shift), this regeneration allows us to argue that, if we consider large time intervals $I_1, I_2,\ldots$, then the displacement of $i_s$ in each interval follows approximately the same law.

In many of our proofs, we study the time dual of the multitype contact process. This dual, called the \textit{ancestor process}, was first considered in \cite{neuhauser} and further studied in \cite{valesin}. In these references, it was shown that the ancestor process behaves approximately as a system of coalescing random walks on $\Z$. Because of this, our proofs of Propositions \ref{prop:fdd} and \ref{prop:tight} are inspired in arguments that apply to coalescing random walks and the voter model, an interacting particle system whose dual is (exactly) equal to coalescing random walks. In particular, a key estimate for the proof of Proposition \ref{prop:tight} (see Lemma \ref{cla4onevstun}) was inspired in an argument by Rongfeng Sun for coalescing random walks (\cite{sun}).

\section{Background on the contact process}
\label{s:back}

\subsection{Notation on sets and configurations}
Given a set $A$, we denote by $\#A$ its cardinality and by $\mathds{1}_A$ its indicator function.

We will reserve the letter $\zeta$ to denote elements of $\{0,1\}^{\Z}$, as well as the one-type contact process, and the letter $\xi$ for elements of $\{0,1,2\}^{\Z}$ and the multitype process. We denote by $\underline{0}$ the configuration in which every vertex is in state 0. We write $\{\xi = i\} = \{x \in \Z: \xi(x) = i\}$ (and similarly for $\zeta$). Given $A \subset \Z$, ``$\xi \equiv i$ on $A$'' means that $\xi(x) = i$ for all $x \in A$ (and similarly for $\zeta$).

Throughout the paper, we fix the parameters $R \in \N$ and $\lambda > \lambda_c(R, \N)$. All the processes we will consider will be defined from these two parameters.

\subsection{One-type contact process}
We will now briefly survey some background material on the (one-type) contact process. A \textit{graphical construction} or \textit{Harris system} is a family of independent Poisson processes on $[0,\infty)$,
\begin{equation} \label{eq:def_harris}H = \left((D^x)_{x\in\Z},\;(D^{x,y})_{\substack{x,y\in\Z^d,\;0 < |x-y|\leq R}}\right), \qquad \begin{array}{l}\text{each }D^x \text{ with rate } 1,\\\text{each } D^{x,y} \text{ with rate } \lambda.\end{array} \end{equation}
We view each of these processes as a random discrete subset of $[0,\infty)$. An arrival at time $t$ of the process $D^x$ is called a \textit{recovery mark} at $x$ at time $t$, and an arrival at time $t$ of the process $D^{x,y}$ is called an \textit{arrow} or \textit{transmission} from $x$ to $y$ at time $t$. This terminology is based on the usual interpretation that is given to the contact process, namely: vertices are individuals, individuals in state 1 are \textit{infected} and individuals in state 0 are \textit{healthy}. Although we will focus mostly on the multitype contact process, which we see as a model for competition rather than the spread of an infection, we will still use some infection-related terminology that comes from the study of the classical process.

We will sometimes need to consider restrictions of $H$ to time intervals, and also translations of $H$. We hence introduce the following notation, for $t > 0$ and $z \in \Z$:
\begin{equation}\begin{split}&D^x_{[0,t]} = D^x \cap [0,t],\quad D^x \circ\theta(z,t) = \{s-t: s\in D^{x-z},\;s\geq t\},\\[.2cm]
&D^{x,y}_{[0,t]} = D^{x,y} \cap [0,t],\quad D^{x,y} \circ\theta(z,t) = \{s-t: s\in D^{x-z,y-z},\;s\geq t\},\\[.2cm]
&H_{[0,t]} = \left((D^x_{[0,t]})_{x\in\Z},\;(D^{(x,y)}_{[0,t]})_{\substack{x,y\in\Z^d,\;0 < |x-y|\leq R}}\right),\\[.2cm]
&H \circ \theta(z,t) = \left((D^x\circ \theta(z,t))_{x\in\Z},\;(D^{(x,y)}\circ \theta(z,t))_{\substack{x,y\in\Z^d,\;0 < |x-y|\leq R}}\right).
\label{eq:harris_sub_interval}
\end{split}\end{equation}

Given a (deterministic or random) initial configuration $\zeta_0$ and a Harris system $H$, it is possible to construct the contact process $(\zeta_t)_{t\geq 0}$ started from $\zeta_0$ by applying the following rules to the arrivals of the Poisson processes in $H$:
\begin{align}
&\text{if }t \in D^x, \text{ then } \zeta_t = \zeta^{0\to x}_{t-}; \label{eq:rule_single_mark}\\
&\text{if }t \in D^{x,y} \text{ and } \zeta_{t-}(x) = 1, \text{ then } \zeta_t = \zeta^{1 \to y}_{t-},\label{eq:rule_single_arrow}
\end{align}
where $\zeta^{i\to x}$ is defined as in \eqref{eq:i_to_x}. That this can be done in a consistent manner, and that it yields a Markov process with the desired infinitesimal generator, is a non-trivial result which (as the other statements in this section) the reader can find in \cite{lig99}.

Given $x,y \in \mathbb{Z}$, $t'>t\geq 0$ and a Harris system $H$, an \textit{infection path} in $H$ from $(x,t)$ to $(y,t')$ is a path $\upgamma:[t,t']\to \mathbb{Z}$ such that
\begin{equation}\label{eq:def_inf_path}\begin{split}
&\upgamma(t) = x,\qquad \upgamma(t')=y,\\& s \notin D^{\upgamma(s)} \text{ for all } s\in[t,t']\text{ and } s \in D^{\upgamma(s-),\upgamma(s)} \text{ whenever } \upgamma(s-) \neq \upgamma(s)\end{split}\end{equation} 
In case there is an infection path from $(x,t)$ to $(y,t')$, we write $(x,t)\leftrightarrow (y,t')$ in $H$ (or simply $(x,t)\leftrightarrow (y,t')$ if $H$ is clear from the context). Given sets $A, B \subset \Z$, and $I_1, I_2 \subset [0,\infty)$, we write $A \times I_1 \leftrightarrow B \times I_2$ if $(x,t) \leftrightarrow (y,t')$ for some $(x, t)\in A \times I_1$ and $(y,t') \in B \times I_2$. We will also write $A \times I_1 \leftrightarrow (y,t')$ if $(x,t) \leftrightarrow (y, t')$ for some $(x,t)\in A \times I_1$, and similarly $(x,t) \leftrightarrow B \times I_2$. Finally, we convention to put $(x,t) \leftrightarrow (x,t)$.

 When $(\zeta_t)$ is constructed with the rules \eqref{eq:rule_single_mark} and \eqref{eq:rule_single_arrow}, we have
\begin{equation}\label{eq:inf_inf_path}
\zeta_t(x) = \mathds{1}_{\{\{\zeta_0 =1\} \times \{0\} \leftrightarrow (x,t)\}} \text{ for all } x \in \Z,\; t \geq 0.
\end{equation}

We will always assume that the contact process is constructed from a Harris system (this will also be the case for the multitype contact process, which, as we will explain shortly, can be constructed from the same $H$ as the one given above). Additionally, we will often consider more than one process at a given time, and implicitly assume that all the processes are built in the same probability space, using the same $H$.

Let us now list a few facts and estimates that we will need.  By a simple comparison with a Poisson process, we can show that there exist $\kappa, c > 0$ (depending on $\lambda$ and $R$) such that
\begin{equation}\label{eq:behM}
\mathbb{P}\left[(x,r) \leftrightarrow [x-\kappa t, x+\kappa t]^c\times \{r+t\}\right] \leq e^{-ct} \text{ for any } x\in \Z,\; r \geq 0,\; t > 0.
\end{equation}
Given $A \subset \Z$ and $t \geq 0$, define 
$$T^{A \times \{t\}} = \sup\{t' \geq t: A \times \{t\} \leftrightarrow \Z \times \{t'\}\}.$$
We write $T^{(x,t)}$ instead of $T^{\{x\}\times\{t\}}$ and, in case $t = 0$, we omit it and write $T^A$ and $T^x$. By \eqref{eq:phase_transition} and the assumption that $\lambda > \lambda_c$,
$$\P[T^0 = \infty] > 0.$$
In case $T^{A\times \{t\}} = \infty$, we write $A \times \{t\} \leftrightarrow \infty$. Similarly, when $T^{(x,t)} = \infty$, we write $(x,t) \leftrightarrow \infty$.

In Theorem 2.30 in \cite{lig99}, we find that there exists $c > 0$ (depending on $R$ and $\lambda$) such that
\begin{equation} \label{eq:behTT}
\mathbb{P}\left[t< T^0 <\infty  \right] \leq e^{-ct} \text{ for any }t > 0
\end{equation}
and
\begin{equation}\label{eq:many_die}
\P\left[T^A < \infty\right] < e^{-c\#A} \text{ for any } A \subset \Z,\; A\neq \varnothing.
\end{equation}
In the mentioned theorem, these estimates are obtained for the case $R = 1$, but the method of proof is a comparison with oriented percolation that works equally well for $R > 1$.

In contrast, the following coupling result has a straightforward proof for $R = 1$ but is much harder for $R > 1$.
\begin{lemma}
\label{lem:couple_ones}
There exists $\bar \beta > 0$ such that the following holds. For any $\varepsilon > 0$ there exists $S_0 > 0$ such that, if $S \ge S_0$ and  $\zeta_0,\;\zeta_0'$ satisfy, for some $a < b$,
\begin{align*}
&\zeta_0(x) = \zeta_0'(x) \text{ for all } x \in [a,b];\\[.2cm]
&\zeta_0(x) \neq 0 \text{ for all } x \in [a-S, a) \cup (b, b+S];\\[.2cm]
&\zeta'_0(x) \neq 0 \text{ for all } x \in (-\infty, a) \cup (b, \infty)
\end{align*}
and $(\zeta_t)$, $(\zeta'_t)$ are contact processes started from $\zeta_0$ and $\zeta'_0$ and constructed with the same Harris system, then, with probability larger than $1-\varepsilon$, 
\begin{equation}
\zeta_t(x)=\zeta'_t(x) \text{ for all } x \in [a - \bar\beta t,\;b + \bar\beta t] \text{ and } t \geq 0.
\end{equation}
\end{lemma}
The proof follows from Proposition 2.7 in \cite{ampv} (see also the treatment of the event $\mathcal{H}_2$ in page 11 of that paper). The key idea is an event which the authors called the formation of a \textit{descendancy barrier}; this means that in a space-time set of the form $\mathcal{C}_x = \{[x-\bar \beta t,x+\bar \beta t]: t \geq 0\}$, every vertex that is reachable by an infection path from $\Z \times \{0\}$ is reachable by an infection path from $(x,0)$. For the statement of the present lemma, it would suffice to argue that, if $S$ is large, with high probability one can find $x \in [a-S,a)$ and $y \in (b, b+S]$ so that a descendancy barrier is formed from both $x$ and $y$.

\begin{remark}
The above lemma also holds, with the same proof, for $a = -\infty$ or $b = \infty$.
\end{remark}

\begin{lemma}\label{lem:desc_bar_sides} For any $\varepsilon > 0$ and $\sigma, \sigma'$ with $-\bar\beta \le \sigma < \sigma'\le \bar \beta$ there exists $S_0 >0$  such that, if $S \ge S_0$ and $\zeta_0(x) = 1$ for all $x \in [-S,S]$, then with probability larger than $1-\varepsilon$,
\begin{equation}
\text{for all $t \geq 0$, there exists } x \in [-S + \sigma t,\;S+\sigma't] \text{ such that } \zeta_t(x) = 1.
\end{equation}
\end{lemma}
\begin{proof}
By Lemma \ref{lem:couple_ones}, it suffices to prove that, given $\varepsilon, \sigma, \sigma'$, there exists $S_0$ so that, for $S \ge S_0$,
\begin{equation*}
\P\left[\text{for all $t \ge 0$, there exists $x \in [-S + \sigma t,\;S + \sigma't]$ with } \zeta^\Z_t(x) = 1\right] > 1-\varepsilon,
\end{equation*}
where $(\zeta^\Z_t)$ denotes the one-type contact process started from full occupancy. For any $t$ we have
\begin{align*}
\P\left[\zeta^\Z_t \equiv 0 \text{ on } [-S + \sigma t,\;S + \sigma't]\right]&=\P\left[\Z\times \{0\} \nleftrightarrow (x,t) \text{ for all } x \in [-S+\sigma t,\;S + \sigma't]\right]\\[.2cm]
&=\P\left[(x,0) \nleftrightarrow \Z\times \{0\} \text{ for all } x \in [-S + \sigma t,\;S + \sigma't]\right]\\[.2cm]&\leq \P\left[T^{[-S+\sigma t,\;S+\sigma' t]} < \infty\right] \stackrel{\eqref{eq:many_die}}{<} e^{-2cS+ c(\sigma'-\sigma) t}.
\end{align*}
Also, using the Strong Markov Property it is easy to verify that, for some $\delta > 0$ that depends on $\sigma, \sigma', \lambda$ and $R$ but not on $t$,
\begin{equation*}\P\left[\zeta^\Z_t \equiv 0 \text{ on } [-S +\sigma t,\; S + \sigma't\right] > \delta \cdot \P\left[\zeta^\Z_s \equiv 0 \text{ on } [-S +\sigma s,\; S + \sigma's] \text{ for some } s \in [t-1, t]\right]\end{equation*}
for all $t \geq 1$. In conclusion,
\begin{align*}
&\P\left[\text{there exists } t > 0: \zeta^\Z_t \equiv 0 \text{ on } [-S + \sigma t,\;S+\sigma t]\right] \\&\leq \sum_{t=1}^\infty \P\left[\text{there exists } s \in [t-1, t]: \zeta^\Z_t \equiv 0 \text{ on } [-S + \sigma s,\;S + \sigma's] \right]\leq \frac{1}{\delta} \sum_{t=1}^\infty e^{-2cS + c(\sigma'-\sigma)t},
\end{align*}
which is smaller than $\varepsilon$ if $S$ is large enough.
\end{proof}

\subsection{Multitype contact process}
\label{ss:mcp}
\textbf{Graphical construction.} Due to our choice of parameters in \eqref{eq:parameters}, it is possible to construct the multitype contact process with the same graphical construction $H$ as the one we have given in  \eqref{eq:def_harris} for the single-type process. The effects of recovery marks and arrows are:
\begin{align}
&\text{if }t \in D^x, \text{ then } \xi_t = \xi^{0\to x}_{t-}; \label{eq:rule_multi_mark}\\
&\text{if }t \in D^{x,y}, \; \xi_{t-}(x) = i \text{ and } \xi_{t-}(y) = 0, \text{ then } \xi_t = \xi^{i \to y}_{t-},\qquad i=1,2,\label{eq:rule_multi_arrow}
\end{align}
where $\xi^{i\to x}$ is defined in \eqref{eq:i_to_x}. These rules lead to the correct transition rates, as prescribed in  \eqref{eq:generator} and \eqref{eq:parameters}.

It will often be convenient to construct several processes, one-type or multitype or both, in the same probability space and using a single realization of $H$. When we do so, the following will be quite useful.
\begin{claim}If $(\xi_t)_{t \geq 0}$, $(\xi'_t)_{t \geq 0}$ are constructed with the same Harris system, 
\begin{equation}\label{eq:attract_multi}
\begin{split}
\{\xi_0 = 2\} \subseteq \{\xi'_0 = 2\}\text{ and } \{\xi_0 = 1\} \supseteq \{\xi'_0 = 1\} \text{ implies}\\
\{\xi_t = 2\} \subseteq \{\xi'_t = 2\}\text{ and }\{\xi_t = 1\} \supseteq \{\xi'_t = 1\} \text{ for all $t$}.
\end{split}
\end{equation}
\end{claim}
\begin{proof}
Simply consider the partial order 
$$\xi \prec \xi'  \quad \Leftrightarrow \quad \{\xi = 2\} \subseteq \{\xi' = 2\}\text{ and } \{\xi = 1\} \supseteq \{\xi' = 1\}$$
and note that the rules \eqref{eq:rule_multi_mark} and \eqref{eq:rule_multi_arrow} preserve this order.
\end{proof}

We will keep using the infection paths of $H$, as defined in \eqref{eq:def_inf_path}. We can obtain
\begin{equation}\label{eq:inf_inf_multi_path}
\xi_t(x) \neq 0 \text{ if and only if } \{\xi_0 \neq 0\}\times\{0\} \leftrightarrow (x,t)
\end{equation}
from \eqref{eq:rule_multi_mark} and \eqref{eq:rule_multi_arrow} similarly to how \eqref{eq:inf_inf_path} is obtained from \eqref{eq:rule_single_mark} and \eqref{eq:rule_single_arrow}. In particular,
\begin{equation}\label{eq:colorblind}
\begin{split}\left( \mathds{1}_{\{\xi_t(x) \neq 0\}}: x \in \Z\right)_{t\ge 0} \text{ has same distribution as a one-type contact }\\\text{process started from } \left(\mathds{1}_{\{\xi_0(x) \neq 0\}}: x \in \Z\right).\end{split}
\end{equation}
This is quite convenient, but there is a drawback: in case $\xi_t(x) \neq 0$, it is not so simple to deduce its value from $\xi_0$ and the infection paths in $H$.  That is because an infection path is only successful in carrying type $i \in \{1,2\}$ if all of the path's arrows land on space-time points that are in state 0. To deal with this, we will need some extra definitions.

Given a realization of $H = ((D^x)_{x\in \Z}, (D^{x,y})_{x,y\in\Z,\;0<|x-y|\leq R})$ and $A \subset \Z$, we define 
\begin{equation}\begin{split}
H_A = ((D^x)_{x\in \Z}, (D_A^{x,y})_{x,y\in\Z,\;0<|x-y|\leq R}),\text{ where }\\
D_A^{x,y} = \{t \in D^{x,y}: A\times \{0\} \nleftrightarrow (y,t) \text{ in } H\}.
\end{split}\end{equation} 
Obviously, any infection path of $H_A$ is also an infection path of $H$. It is also readily seen that the one-type process with initial occupancy in $A$ is the same whether it is built using $H$ or $H_A$, that is,
$$\zeta^A_t(H) = \zeta^A_t(H_A),\quad t \geq 0.$$
Additionally, in $H_A$, the number of infection paths from $A \times \{0\}$ to any $(x,t)$ is either 0 or 1 (corresponding respectively to the cases $[\zeta^A_t(H_A)](x) = [\zeta^A_t(H)](x) = 0$ and $[\zeta^A_t(H_A)](x)= [\zeta^A_t(H)](x) = 1$). In the second case, the unique infection path in $H_A$ from $A\times\{0\}$ to $(x,t)$ is denoted $\upgamma^*_{A,x,t}:[0,t]\to\Z$. We can also characterize $\upgamma^*_{A,x,t}$ as the unique infection path in $H$ satisfying
\begin{equation}\begin{split}&\upgamma^*_{A,x,t}(0) \in A,\quad \upgamma^*_{A,x,t}(t) = x,\\&A\times\{0\} \nleftrightarrow (\upgamma^*_{A,x,t}(s),s-) \text{ whenever } \upgamma^*_{A,x,t}(s-) \neq \upgamma^*_{A,x,t}(s).\label{eq:good_gamma}\end{split}\end{equation}

Now fix $\xi_0 \in \{0,1,2\}^{\Z}$ and assume $(\xi_t)_{t \ge 0}$ is the multitype contact process started from $\xi_0$ and constructed with a Harris system $H$. We claim that
\begin{equation}\label{eq:claim_about_multi_constr}
\xi_t(x) = \mathds{1}_{\{\{\xi_0 \neq 0\}\times \{0\} \leftrightarrow (x,t)\}} \cdot \xi_0\big(\upgamma^*_{\{\xi_0 \neq 0\},x,t}(0) \big).
\end{equation}
Indeed, if the right-hand side is zero, then the indicator function is zero (as the other term is non-zero by construction), so $\xi_t(x) = 0$ holds by \eqref{eq:inf_inf_multi_path}. If the right-hand side of \eqref{eq:claim_about_multi_constr} is non-zero, then the definition of infection paths together with \eqref{eq:rule_multi_arrow}, \eqref{eq:inf_inf_multi_path} and \eqref{eq:good_gamma} imply that  $\xi_s(\upgamma^*_{\{\xi_0\neq 0\},x,t}(s)) = \xi_0(\upgamma^*_{\{\xi_0\neq 0\},x,t}(0))$ for every $s \in [0,t]$, so the equality also follows.

%As a consequence, if we use the same Harris system $H$ to construct a one-type process $(\zeta_t)_{t\geq 0}$ started from $\zeta_0 = \{\xi \neq 0\}$ (that is, $\zeta_0(x) = 0$ if and only if $\xi_0(x) = 0$), we get
%$$\xi_t(x) = 0 \text{ if and only if } \zeta_t(x) = 0,\qquad t\geq 0.$$
%$(\xi_t)_{t\geq 0}$ can thus be seen as a ``coloring'' of $(\zeta_t)_{t \geq 0}$.

These considerations are summarized as follows:

\begin{lemma}
Let $(\xi_t)_{t \geq 0}$ be the multitype contact process started from a fixed $\xi_0 \in \{0,1,2\}^\Z$ and constructed with a Harris system. Then,
 \begin{equation}\label{eq:best_path}
\xi_t(x) = i \quad \text{ if and only if } \quad \begin{array}{c}  \text{there exists an infection path $\upgamma: [0,t]\to\Z^d$}\\\text{such that } \xi_0(\upgamma(0)) = i,\; \upgamma(t) = x  \text{ and }\\  \xi_{s-}(\upgamma(s)) = 0 \text{ whenever } \upgamma(s) \neq \upgamma(s-)\end{array},\quad i = 1,2.
\end{equation}
Moreover, there exists at most one infection path satisfying the stated properties.\end{lemma}\vspace{.4cm}

\noindent \textbf{Ancestry process.}
We now define an auxiliary process that is key in making the graphical construction of the multitype contact process more tractable. Again fix a Harris system $H$ and let $(x,r) \in \Z \times [0,\infty)$. Given $t > r$, by arguing similarly to how we did in the previous paragraphs, it can be shown that 
$$(x,r) \leftrightarrow \Z\times\{t\} \quad\text{ if and only if }\quad \begin{array}{c}\text{ there exists a unique infection path}\\ \uppsi:[r,t] \to \Z\text{ such that } \uppsi(r) = x \text{ and}\\ (\uppsi(s-),s)\nleftrightarrow \Z\times \{t\} \text{ whenever } \uppsi(s-) \neq \uppsi(s).\end{array}$$
In case it exists, we denote this path by $\psi^*_{x,r,t}$, or $\psi^*_{x,r,t}(H)$ when we want to make the dependence on the Harris system explicit. Note that $\psi^*_{x,r,t}$ only depends on $H \cap [r,t]$.

We claim that, for $r < t < t'$, 
\begin{equation}
\label{eq:prop_ancestry_path}
\text{if } \uppsi^*_{x,r,t}(t) = y \text{ and } (y,t) \leftrightarrow \Z \times \{t'\}, \text{ then } \uppsi^*_{x,r,t'}(s) = \uppsi^*_{y,t,t'}(s) \text{ for all } s \in [t, t'].
\end{equation}
Indeed,  defining $\uppsi: [r,t'] \to \Z$ by
$$\uppsi(s) = \left\{ \begin{array}{ll} \uppsi^*_{x,r,t}(s) & \text{if } s \in [r,t];\\[.2cm]\uppsi^*_{y,t,t'}(s) &\text{if } s \in [t, t'],\end{array}\right.$$
we have that 
\begin{itemize}
\item if $s \in [r,t]$ and $\uppsi(s-) \neq \uppsi(s)$, then by the definition of $\uppsi^*_{x,r,t}$, we have that $(\uppsi(s-),s) \nleftrightarrow \Z \times \{t\}$, so $(\uppsi(s-),s) \nleftrightarrow \Z \times \{t'\}$;
\item if $s \in [t,t']$ and $\uppsi(s-) \neq \uppsi(s)$, then by the definition of $\uppsi^*_{y,t,t'}$, we have that $(\uppsi(s-),s) \nleftrightarrow \Z \times \{t'\}$,
\end{itemize}
so that, by the uniqueness of $\uppsi^*_{x,r,t'}$, we get $\uppsi = \uppsi^*_{x,r,t'}$, so \eqref{eq:prop_ancestry_path} follows.

We define, for $x\in \Z$ and $r \geq 0$,
\begin{equation}\label{eq:def_ancestor_process}
\eta^{(x,r)}_t = \left\{\begin{array}{ll}\uppsi^*_{x,r,t}(t),&\text{if } t \in [r, T^{(x,r)});\\[.2cm]\triangle &\text{otherwise,}\end{array}\right.
\end{equation}
where $\triangle$ is interpreted as a ``cemetery'' state. The process $(\eta^{(x,r)}_t)_{t \geq r}$ is called the \textit{ancestor process} of $(x,r)$. In case $r = 0$, we write $\eta^x_t$ instead of $\eta^{(x,0)}_t$, and in case $r = x = 0$, we omit the superscript and write $\eta_t$. Naturally,
\begin{equation}\label{eq:eqdisteta}\big(\eta^{(x,r)}_{r+t} - x\big)_{t \geq 0} \stackrel{\text{(dist.)}}{=} \left(\eta_t\right)_{t\geq 0}.\end{equation}

Now \eqref{eq:prop_ancestry_path} can be rewritten as
\begin{equation}
\label{eq:prop_ancestry_sum}
\text{if }r < t < t',\;\eta^{(x,r)}_t \neq \triangle \text{ and } \eta^{(\eta^{(x,r)}_t,t)}_{t'} \neq \triangle,\text{ then }\eta^{(x,r)}_{t'} = \eta^{(\eta^{(x,r)}_t,t)}_{t'}.
\end{equation}
In particular, we get
\begin{equation}
\label{eq:prop_ancestry_inf}
\text{if }r < t < \infty,\;\eta^{(x,r)}_t \neq \triangle \text{ and } T^{(\eta^{(x,r)}_t,t)} = \infty,\text{ then }\eta^{(x,r)}_{t'} = \eta^{(\eta^{(x,r)}_t,t)}_{t'} \text{ for all } t' \geq t.
\end{equation}\vspace{.2cm}

\noindent \textbf{Joint construction of primal and dual processes.} We now explain the relationship between the multitype contact process and the ancestor process. Given a Harris system $H=((D^x),(D^{x,y}))$ and $t > 0$, we recall the notation introduced in \eqref{eq:harris_sub_interval} and define the \textit{reversed Harris system} $\hat H_{[0,t]}$ by
\begin{equation}
\begin{split}
&\hat H_{[0,t]} = \left((\hat D^x_{[0,t]})_{x \in \Z},\;(\hat D^{x,y}_{[0,t]})_{x,y\in\Z,\;0<|x-y|\le R}\right),\text{ where} \\[.2cm]
&\hat D^x_{[0,t]} = \{t - s: s \in D^x \cap [0,t]\} \text{ and }\hat D^{x,y}_{[0,t]} = \{t - s: s \in D^{y,x} \cap [0,t]\}.
\end{split}
\end{equation}
In words, $\hat H_{[0,t]}$ is the Harris system on the time interval $[0,t]$ obtained from $H \cap [0,t]$ by reversing time and reversing the direction of the arrows.

Assume we are given $\xi_0 \in \{0,1,2\}^\Z$ and construct $(\xi_t)$ started from $\xi_0$ using the Harris system $H$. Fix $t > 0$ and  assume that we use $\hat H_{[0,t]}$ to  construct the ancestor processes $$\eta^x_s = \eta^x_s(\hat H_{[0,t]}): 0 \leq s \leq t,\; x\in \Z.$$
One immediate consequence of this joint construction is that
\begin{equation}\text{if } \eta^x_t = \triangle, \text{ then } \Z \times \{0\} \nleftrightarrow (x,t) \text{ in } H, \text{ hence } \xi_t(x) = 0.\label{eq:dual_eq_dies}\end{equation}
More interestingly, by \eqref{eq:best_path} and the definition of the ancestor process, 
\begin{equation}
\text{if } \eta^x_t \neq \triangle \text{ and } \xi_0(\eta^x_t) \neq 0,\text{ then } \xi_t(x) = \xi_0(\eta^x_t).\label{eq:dual_eq_surv}
\end{equation}
Indeed, the $\hat H_{[0,t]}$-infection path $\uppsi^*_{x,0,t}$, when ran backwards and with arrows reversed, corresponds exactly to the $H$-infection path $\upgamma^*_{\Z,x,t}$.

As a consequence of these considerations, we have
\begin{claim}
If $\xi_0(x) \neq 0$ for all $x \in \Z$, then, with the convention that $\xi_0(\triangle) = 0$,
\begin{equation} \label{eq:duality_equation_multi}
\left( \xi_t(x): x \in \mathbb{Z}\right) \stackrel{(\text{dist.})}{=} \left(\xi_0(\eta^x_t): x\in\mathbb{Z}\right). \end{equation}
\end{claim}

\subsection{Renewal times of the ancestry process}
We now recall the renewal structure from which we are able to decompose the ancestor process into pieces that are independent and identically distributed. This then allows us to find an embedded random walk in $(\eta_s)$ and argue that the whole of the trajectory of $(\eta_s)$ remains close to this embedded random walk. Most of the results of this subsection are not new (they appear in \cite{neuhauser} or \cite{valesin} or both); in an effort to balance the self-sufficiency of this paper with shortness of exposition, we will include a few key proofs and omit others.

\begin{lemma}
\label{lem:no_renewals_ab}
There exists $c > 0$ such that, for any $b > a \geq 0$, we have
\begin{equation}\label{eq:behMM}\P\left[ \eta_s \neq \triangle \text{ and } (\eta_s,s)\nleftrightarrow \infty \text{ for all } s \in [a,b] \right] < e^{-c(b-a)}.\end{equation}
\end{lemma}
The proof is a simple consequence of \eqref{eq:behTT}; see Proposition 1, page 474, of \cite{neuhauser}.

Given $A \subset \Z$, we write
\begin{equation}\label{eq:notation_tilde_P}
\tilde \P^A\left[\;\cdot\; \right] = \P[\;\cdot\;|\;(y,0)\leftrightarrow \infty \text{ for all } y \in A].
\end{equation}
In case $A = \{x\}$, we write $\tilde \P^x$ instead of $\tilde \P^{\{x\}}$ and in case $x = 0$, we omit the superscript.
\begin{lemma}\label{lem_x_goes_to_y_first}
Let $t_0 > 0$ and
$$\uptau = \inf\{t \geq t_0: \eta_t \neq \triangle \text{ and }(\eta_t, t) \leftrightarrow \infty\}. $$
For any $y \in \Z$ and events $E, F$ on Harris systems,
\begin{equation}\label{eq:main_prop_H_uptau}\begin{split}&\P\left[\uptau < \infty,\;H_{[0,\uptau]} \in E,\;\eta_\uptau = y\text{ and } H \circ \theta(0,\uptau) \in F \right] \\[.2cm]
&=  \P\left[\uptau < \infty,\; H_{[0,\uptau]} \in E,\;\eta_\uptau = y\right]
\cdot \tilde \P^{y} \left[H \in F\right].
\end{split}\end{equation}
\end{lemma}
\begin{proof}
We let $\sigma_0 = t_0$ and, for $k \geq 0$, define $\sigma_{k+1}$ as follows:
$$
\sigma_{k+1} = \begin{cases} T^{\left(\eta_{\sigma_k},\sigma_k\right)}&\text{if } \sigma_k < \infty \text{ and } \eta_{\sigma_k} \neq \triangle\\[.2cm]
\sigma_k&\text{if }\sigma_k < \infty \text{ and }\eta_{\sigma_k} = \triangle;\\[.2cm]
\infty &\text{if } \sigma_k = \infty. \end{cases}
$$
$(\sigma_k)_{k=0}^\infty$ is thus an increasing sequence of stopping times with respect to the sigma-algebra of Harris systems. We note that, in case we have $\sigma_k < t < \sigma_{k+1} < \infty$, then \eqref{eq:prop_ancestry_sum} gives $\eta^{(\eta_{\sigma_k},\sigma_k)}_{t'} = \eta^{(\eta_t,t)}_{t'}$ for all $t' \in [t, T^{(\eta_t,t)})$. So we have
\begin{equation}
\text{if } \sigma_k < t < \sigma_{k+1} < \infty, \text{ then } T^{(\eta_t,t)} \leq T^{(\eta_{\sigma_k},\sigma_k)} = \sigma_{k+1}.
\end{equation}
As a consequence, we obtain
\begin{equation}\label{eq:claim_uptau}\{\uptau < \infty\} = \cup_{k=0}^\infty \{\uptau = \sigma_k < \infty,\; \sigma_{k+1} = \infty\}.\end{equation}

Using \eqref{eq:claim_uptau}, the left-hand side of \eqref{eq:main_prop_H_uptau} becomes
\begin{align*}
&\sum_{k=0}^\infty \P\left[\sigma_k < \infty;\; H_{[0,\sigma_k]} \in E;\;\eta_{\sigma_k} = y \text{ and } (y,\sigma_k) \leftrightarrow \infty;\; H\circ \theta(0,\sigma_k) \in F \right]\\
&= \tilde \P^{y}\left[ H \in F\right] \cdot \sum_{k=0}^\infty \P\left[\sigma_k < \infty;\;H_{[0,\sigma_k]} \in E;\;\eta_{\sigma_k} = y \text{ and } (y,\sigma_k)\leftrightarrow \infty \right]\\
&= \tilde \P^{y}\left[ H \in F\right] \cdot \P\left[ \uptau < \infty;\; H_{[0,\uptau]} \in E;\;\eta_\uptau = y \right].
\end{align*}

\end{proof}

Given $(z,r) \in \Z \times [0,\infty)$, on the event $(z,r)\leftrightarrow \infty$ we define the times
$$\uptau^{(z,r)}_0 = r,\qquad \uptau^{(z,r)}_k = \inf\{t \geq \uptau^{(z,r)}_{k-1} + 1: (\eta^{(z,r)}_t,t)\leftrightarrow \infty \},\; k \geq 1.$$
We write $\uptau^z_k$ instead of $\uptau^{(z,0)}_k$ and $\uptau_k$ instead of $\uptau^0_k$.
We now state three simple facts about these random times. First, it follows from \eqref{eq:prop_ancestry_inf} that 
\begin{equation}\label{eq:aid_uptau}
\text{if } \uptau_1 = t \text{ and } \eta_t = z, \text{ then } \uptau_k = \uptau^{(z,t)}_{k-1} \text{ for all } k \geq 1.
\end{equation}
Second, from \eqref{eq:behMM} it is easy to obtain
\begin{equation}\label{eq:tau_is_fin}
\P\left[ \uptau_k < \infty \text{ for all }k\mid(0,0) \leftrightarrow \infty\right] = 1.
\end{equation}
Third, by putting \eqref{eq:behM} and \eqref{eq:behMM} together, it is easy to show that
\begin{equation}\label{eq:bhT}
\tilde \P\left[\max\left(\uptau_1,\;\sup_{0\leq s \leq \uptau_1} |\eta_s|\right) > t\right] \leq e^{-ct},\qquad t > 0.
\end{equation}

Our main tool in dealing with the ancestor process is the following result.
\begin{proposition}\label{prop:rmn}
\begin{enumerate}
\item Under $\tilde \P$, 
\begin{equation} \label{eq:inc_iid}\left(\uptau_{k+1} - \uptau_k,\;(\eta_t - \eta_{\uptau_k})_{\uptau_k \leq t <\uptau_{k+1}}\right),\; k = 1, 2, \ldots \text{ are i.i.d.} \end{equation} In particular, $(\eta_{\uptau_k})_{k \geq 0}$ is a random walk on $\Z$ with increment distribution
$$\tilde\P\left[\eta_{\uptau_{k+1}} = w\;|\;\eta_{\uptau_k} =z\right] = \tilde \P\left[\eta_{\uptau_1} = w-z\right].$$ 
\item There exist $C, c > 0$ such that, for any $t \geq 0$, $r > 0$ and $x \geq 0$,
\begin{equation}
\label{eq:bondRW} \P\left[\eta_{t+ r} \neq \triangle,\; \sup_{s \in [t,t+r]}|\eta_s - \eta_t| > x\right] \leq Ce^{-cx^2/r} + Cre^{-c|x|}. 
\end{equation}
\item Under $\tilde \P$, \begin{equation}\label{eq:conv_norm} \frac{\eta_t}{\sqrt{t}} \overset{t\to \infty}{\underset{\text{(dist.)}}\longrightarrow} \mathcal{N}(0,\sigma^2) \text{ with }\sigma > 0.\end{equation}
\end{enumerate}
\end{proposition}
\begin{proof}
A proof of part \eqref{eq:inc_iid} can be found in \cite{neuhauser}, but we give another one here. Let $E_0, \ldots, E_k$ be measurable subsets of $\cup_{t \geq 0} D[0,t]$, the space of finite-time trajectories that are right-continuous with left limits. We evaluate
\begin{align}
\nonumber&\tilde \P\left[ \left(\eta_s - \eta_{\uptau_i}\right)_{\uptau_i < s \leq \uptau_{i+1}} \in E_i \text{ for } i = 0, \ldots, k\right]\\[.2cm]
&= \P\left[(0,0)\leftrightarrow \infty\right]^{-1} \cdot \sum_{z \in \Z} \P\left[\begin{array}{l}(0,0) \leftrightarrow \infty ,\; \left(\eta_s\right)_{0 \leq s \leq \uptau_1} \in E_0,\;\eta_{\uptau_1} = z,\\[.2cm]\left(\eta_s - \eta_{\uptau_i}\right)_{\uptau_i < s \leq \uptau_{i+1}} \in E_i \text{ for } i = 1, \ldots, k\end{array}\right]\nonumber\\[.2cm]
&\stackrel{\eqref{eq:tau_is_fin}}{=} \P\left[(0,0)\leftrightarrow \infty\right]^{-1} \cdot \sum_{z \in \Z} \P\left[\begin{array}{l}\uptau_1 < \infty ,\; \left(\eta_s\right)_{0 \leq s \leq \uptau_1} \in E_0,\;\eta_{\uptau_1} = z,\\[.2cm]\left(\eta_s - \eta_{\uptau_i}\right)_{\uptau_i < s \leq \uptau_{i+1}} \in E_i \text{ for } i = 1, \ldots, k\end{array}\right]
\label{eq:oaid_uptau}
\end{align}
Now, by \eqref{eq:prop_ancestry_inf} and \eqref{eq:aid_uptau}, 
$$\text{on }\{(0,0)\leftrightarrow  \infty,\; \eta_{\uptau_1} = z\} \text{ we have }\eta_s = \eta^{(z,\uptau_1)}_s \text{ for all } s \geq \uptau_1 \text{ and }\uptau_k = \uptau^{(z,\uptau_1)}_{k-1} \text{ for all }k \geq 1.$$ Applying these identities and Lemma \ref{lem_x_goes_to_y_first}, we obtain that \eqref{eq:oaid_uptau} is equal to
%$$\P\left[(0,0)\leftrightarrow \infty\right]^{-1} \cdot \sum_{z \in \Z} \P\left[\begin{array}{l}(0,0) \leftrightarrow \infty ,\; \left(\eta_s\right)_{0 \leq s \leq \uptau_1} \in E_0,\;\eta_{\uptau_1} = z,\\[.3cm]\left(\eta^{(z,\uptau_1)}_s - \eta^{(z,\uptau_1)}_{\uptau_i}\right)_{\uptau_i < s \leq \uptau_{i+1}} \in E_i \text{ for } i = 1, \ldots, k- 1\end{array}\right].$$
%Applying Lemma \ref{lem_x_goes_to_y}, this expression is equal to
$$\begin{aligned}&\P\left[(0,0)\leftrightarrow \infty\right]^{-1} \cdot \sum_{z \in \Z} \P\left[\uptau_1 < \infty ,\; \left(\eta_s\right)_{0 \leq s \leq \uptau_1} \in E_0,\;\eta_{\uptau_1} = z\right]\\&\qquad\qquad\qquad\qquad\qquad \cdot \tilde \P^z\left[ \left(\eta^{z}_s - \eta^{z}_{\uptau_i}\right)_{\uptau_i < s \leq \uptau_{i+1}} \in E_{i+1} \text{ for } i = 0, \ldots, k- 1\right]\\[.2cm]
&=\P\left[(0,0)\leftrightarrow \infty\right]^{-1} \cdot \sum_{z \in \Z} \P\left[\uptau_1 < \infty ,\; \left(\eta_s\right)_{0 \leq s \leq \uptau_1} \in E_0,\;\eta_{\uptau_1} = z\right]\\&\qquad\qquad\qquad\qquad\qquad \cdot \tilde \P\left[ \left(\eta_s - \eta_{\uptau_i}\right)_{\uptau_i < s \leq \uptau_{i+1}} \in E_{i+1} \text{ for } i = 0, \ldots, k- 1\right]\\[.2cm]
&=\tilde\P\left[\left(\eta_s\right)_{0 \leq s \leq \uptau_1} \in E_0 \right] \cdot \tilde \P\left[ \left(\eta_s - \eta_{\uptau_i}\right)_{\uptau_i < s \leq \uptau_{i+1}} \in E_{i+1} \text{ for } i = 0, \ldots, k- 1\right].
\end{aligned}$$
We now iterate this computation to obtain \eqref{eq:inc_iid}.

For the remaining statements, we will need a definition. On the event $\{(0,0)\leftrightarrow \infty\}$, let
$$ N_t = \max\{k: \uptau_k \leq t\},\quad t \geq 0.$$
Since we have $\uptau_{k+1} - \uptau_k \geq 1$ for all $k$, we obtain
\begin{equation}\label{eq:separation}
N_{t+r} - N_t \leq \lceil r \rceil \text{ for any } t \geq 0,\;r>0.
\end{equation}

We now turn to \eqref{eq:bondRW}. The left-hand side is less than
\begin{align}
\P\left[t+r < T^0 < \infty,\;\sup_{t\leq s \leq t+r} |\eta_s - \eta_t| > x \right] + \tilde \P\left[\sup_{t\leq s\leq t+r} |\eta_s-\eta_t| > x\right].
\label{eq:12auxrw}\end{align}
The first term is less than
$$\P\left[ T^0<\infty,\; \sup_{0\leq s \leq T^0} |\eta_s| > \frac{x}{2}\right] \leq \P\left[\frac{x}{2\kappa} < T^0 < \infty\right] + \P\left[\sup_{0\leq s \leq x/(2\kappa)} |\eta_s| > \frac{x}{2} \right],$$
where $\kappa$ is as in \eqref{eq:behM}. Then, \eqref{eq:behM} and \eqref{eq:behTT} show that the sum is less than $e^{-c|x|}$ for some $c > 0$. The second term in \eqref{eq:12auxrw} is less than 
\begin{align*}&\tilde \P\left[\max_{N_t \leq k \leq N_{t + r}} |\eta_{\uptau_k} - \eta_{\uptau_{N_t}}| > \frac{x}{2} \right] + \tilde \P \left[\max_{N_t \leq k \leq N_{t+r}}\; \sup_{s\in[\uptau_{k}, \uptau_{k+1}]} |\eta_{s} - \eta_{\uptau_{k}}| > \frac{x}{2} \right]\\[.2cm]
&\stackrel{\eqref{eq:inc_iid},\eqref{eq:separation}}{\leq} \tilde  \P\left[\max_{0 \leq k \leq \lceil r \rceil} |\eta_{\uptau_k}| > \frac{x}{2} \right] + \lceil r \rceil \cdot \tilde \P\left[\sup_{s \in [0,\uptau_1]} |\eta_{s}| > \frac{x}{2} \right]\\[.2cm] &\hspace{.5cm}\stackrel{\eqref{eq:bhT}
}{\leq}  \tilde  \P\left[\max_{0 \leq k \leq \lceil r \rceil} |\eta_{\uptau_k}| > \frac{x}{2} \right] + \lceil r \rceil e^{-cx}.
\end{align*}
Now, using standard random walk estimates (see for example Proposition 2.1.2 in \cite{lawler}), we can  bound the first term above by $e^{-cx^2/\lceil r \rceil}$. This completes the proof of \eqref{eq:bondRW}.

Finally, let us prove \eqref{eq:conv_norm}. 
Denote
$$\mu = (\tilde \E \uptau_1)^{-1} \stackrel{\eqref{eq:bhT}}{<} \infty.$$
In Lemma 2.5 in \cite{valesin}, it is shown that
\begin{equation}\label{eq:behNt}
\tilde \P\left[\sup\{|\eta_s - \eta_{\uptau_{N_t}}|: \uptau_{N_t} \leq s \leq \uptau_{N_1 + 1}\} > x \right] \leq e^{-cx}, \qquad t > 0,\; x > 0.
\end{equation}
We write
\begin{equation*}
\frac{\eta_t}{\sqrt{t}} = \frac{\eta_{\lfloor t/\mu \rfloor}}{\sqrt{t}} + \frac{\eta_{N_t} - \eta_{\lfloor t/\mu\rfloor}}{\sqrt{t}} + \frac{\eta_{t} - \eta_{N_t}}{\sqrt{t}}
\end{equation*}
By \eqref{eq:inc_iid} and the Central Limit Theorem, $\frac{\eta_{\lfloor t/\mu \rfloor}}{\sqrt{t}}$ converges in distribution, as $t\to\infty$, to $\mathcal{N}(0,\sigma^2)$ with $\sigma > 0$. Using \eqref{eq:behNt}, we have that $\frac{\eta_{t} - \eta_{N_t}}{\sqrt{t}}$ converges in probability, as $t\to \infty$, to zero. Hence, \eqref{eq:conv_norm} will follow if we prove that the remaining term also satisfies
\begin{equation}\frac{\eta_{N_t} - \eta_{\lfloor t/\mu\rfloor}}{\sqrt{t}} \overset{t\to \infty}{\underset{\text{(prob.)}}\longrightarrow} 0\label{eq:conv_prob_0}\end{equation}With this aim, fix $\varepsilon > 0$. For any $\delta > 0$ we have
\begin{align}\nonumber
&\tilde \P\left[\frac{\eta_{N_t} - \eta_{\lfloor t/\mu\rfloor}}{\sqrt{t}} > \varepsilon\right]\leq \tilde \P\left[\frac{N_t}{t} - \frac{1}{\mu} > \delta \right] \\&+ \tilde \P\left[\frac{\eta_{N_t} - \eta_{\lfloor t/\mu\rfloor}}{\sqrt{t}} > \varepsilon,\; N_t \in \left[ \frac{t}{\mu}-\delta t,\; \frac{t}{\mu}\right]\right] + \tilde \P\left[\frac{\eta_{N_t} - \eta_{\lfloor t/\mu\rfloor}}{\sqrt{t}} > \varepsilon,\; N_t \in \left[ \frac{t}{\mu},\; \frac{t}{\mu} + \delta t\right]\right]\label{eq:almost_conv}
\end{align}
By the Renewal Theorem, $\tilde \P\left[\frac{N_t}{t} - \frac{1}{\mu} > \delta \right] \to 0$ as $t\to \infty$. Next, 
\begin{align*}\tilde \P\left[\frac{|\eta_{N_t} - \eta_{\lfloor t/\mu\rfloor}|}{\sqrt{t}} > \varepsilon,\; N_t \in \left[ \frac{t}{\mu},\; \frac{t}{\mu}+\delta t\right]\right] &\leq \tilde \P\left[\max_{\frac{t}{\mu} \leq i \leq \frac{t}{\mu} + \delta t} |\eta_{\uptau_i} - \eta_{\uptau_{\lfloor t/\mu \rfloor}}| > \varepsilon\sqrt{t}\right] \\[.2cm]&\leq \delta t \frac{\text{Var}(\eta_{\uptau_1})}{\varepsilon^2t} = \delta \frac{\text{Var}(\eta_{\uptau_1})}{\varepsilon^2},\end{align*}
where the last inequality is an application of Kolmogorov's Inequality. The above can be made arbitrarily small by taking $\delta$ small (depending on $\varepsilon$). The other term in \eqref{eq:almost_conv} is then treated similarly, and the proof of \eqref{eq:conv_prob_0} is now complete.
\end{proof}

In \cite{valesin}, results are obtained about the joint behavior of two or more ancestor processes. The method used to obtain such results involved studying renewal times that are more complicated then the $\uptau^x_k$ defined above. We will not present the details here. Rather, let us just mention that, while a single ancestor behaves closely to a random walk (as outlined above), a larger amount of ancestors, when considered jointly, behave closely to a system of coalescing random walks (that is, a system of random walkers that move independently with the added rule that two walkers that occupy the same position merge into a single walker). Taking advantage of this comparison, one can then obtain for ancestor processes several estimates that hold for coalescing random walks. In particular, in Lemma 3.2 in \cite{valesin}, it is shown that
\begin{equation}
\label{eq:pairs_meet}
\text{there exists } C > 0 \text{ such that }\P\left[\eta^x_t,\;\eta^y_t \neq \triangle,\; \eta^x_t \neq \eta^y_t \right] \leq \frac{C|x-y|}{\sqrt{t}},\; x,y \in \Z,\; t > 0.
\end{equation}
Using this result, it is then possible to show that the density of the set of \textit{all} ancestors at time $t$, $\{\eta^x_t: x\in \Z\} \cap \Z$, goes to zero as $t \to \infty$ (see Proposition 3.5 in \cite{valesin}), so that
\begin{equation}
\label{eq:density_goes_to_zero}
\text{for all finite } I \subset \Z,\; \P\left[\{\eta^x_t: x \in \Z\} \cap  I \neq \varnothing\right] \xrightarrow{t \to \infty} 0.
\end{equation}
Finally, we will need the bound
\begin{equation}
\label{eq:key_rw_estimate_rs}
\begin{split}
\text{for any } u > 0 \text{ there exists } C > 0 \text{ such that, for $t$ large enough and any } x < y, \\\P\left[\eta^x_s, \eta^y_s \neq \triangle,\; \eta^x_s > \eta^y_s + u\sqrt{t} \text{ for some } s \leq t \right] < \frac{C}{\sqrt{t}}.
\end{split}
\end{equation}
For coalescing random walks having symmetric jump distribution with finite third moments, this estimate is given by Lemma 2.0.4 in \cite{sun}. As $(\eta^x_t)$ and $(\eta^y_t)$ are not exactly coalescing random walks, the proof of the mentioned lemma has to be adapted to the present context. Given the method of proof of Theorem 6.1 in \cite{valesin}, this adaption does not involve anything new, so we do not include it here.

\subsection{Interface}
Given $\xi\in\{0,1,2\}^\Z$, we write
$$r(\xi) = \sup\{x \in \Z:\xi(x) = 1\},\quad \ell(\xi) = \inf\{x\in\Z: \xi(x) = 2\}.$$
Define
\begin{equation}\label{eq:def_of_Omega}\Omega = \left\{\begin{array}{ll} \xi\in\{0,1,2\}^Z:&\#\{x<0:\xi(x) = 1\} = \#\{x>0:\xi(x) = 2\} = \infty,\\[.2cm]&\#\{x<0:\xi(x) = 2\} < \infty,\;\#\{x>0:\xi(x) = 1\} < \infty  \end{array}\right\};\end{equation}
in particular, $r(\xi) < \infty$ and $\ell(\xi) > -\infty$ for any $\xi \in \Omega$.

As mentioned in the Introduction, $(\xi^h_t)_{t \geq 0}$ denotes the contact process started from the heaviside configuration, \eqref{eq:def_heaviside}, and 
$$r_t = r(\xi^h_t),\qquad \ell_t = \ell(\xi^h_t),\qquad i_t = (r_t + \ell_t)/2,\qquad t \geq 0.$$
The interval delimited by $r_t$ and $\ell_t$ is the \textit{interface}, and $i_t$ is the \textit{interface position}, at time $t$.
Using \eqref{eq:behM}, it is easy to show that, almost surely,
$$\xi^h_t \in \Omega \text{ for all } t \geq 0.$$
It will be useful to have the following rough bound on the displacement of $r_t$ and $\ell_t$. 
\begin{lemma}\label{lem:no_faster}
For any $\varepsilon > 0$ and $\sigma > 0$ there exists $S_0 > 0$ such that, if $S \ge S_0$ and $\xi_0$ satisfies $\xi_0 \equiv 2$ on $(0,\infty)$, then with probability larger than $1 - \varepsilon$,
\begin{equation}\nonumber
\text{for any } t \ge 0,\; r(\xi_t),\ell(\xi_t) \leq S + \sigma t.
\end{equation}
\end{lemma}
\begin{proof}
It is sufficient to prove the result for $\sigma \in (0,\bar\beta)$, where $\bar\beta$ is the constant that appears in Lemmas \ref{lem:couple_ones} and \ref{lem:desc_bar_sides}.
We fix $\sigma', \sigma''$ with 
$$0< \sigma'< \sigma''<\sigma.$$
Using the joint construction of the multitype contact process and the ancestor processes (as described in Subsection \ref{ss:mcp} and in particular equation \eqref{eq:duality_equation_multi}) together with the assumption that $\xi_0 \equiv 2$ on $(0,\infty)$ and Claim \ref{eq:attract_multi}, we have
\begin{align*}
\P\left[\xi_t(x) = 1\right] \leq \P\left[\xi^h_t(x) = 1\right]\leq \P\left[\eta^x_t\neq \triangle,\; \eta^x_t \leq 0 \right].
\end{align*}
If $x \geq 0$, the right-hand side is smaller than or equal to
$$\P\left[\eta^x_t \neq \triangle,\;|\eta^x_t - x| \geq x\right] = \P\left[\eta_t \neq \triangle, \;|\eta_t| \geq x \right] \stackrel{\eqref{eq:bondRW}}{\leq} Ce^{-cx^2/t} + Cte^{-cx}.$$
Combining this with a union bound, we get
\begin{equation}\P\left[r(\xi_t) \geq \frac{S}{3} +\sigma't \text{ for some } t \in \N\right] = \P\left[\xi_t(x) = 1 \text{ for some } t \in \N \text{ and } x \geq \frac{S}{3} + \sigma't\right] < \frac{\varepsilon}{3}\label{eq:bound_before_intervals}\end{equation}
if $S$ is large enough. We then bound
\begin{align*}&\P\left[r(\xi_t) < S/3 + \sigma't\;\text{ and }\; r(\xi_s) \geq 2S/3 + \sigma''t \text{ for some } s \in [t, t+1] \right]\\[.2cm]& \leq \P\left[(-\infty,\;S/3 + \sigma' t) \times \{t\} \leftrightarrow [2S/3 + \sigma'' t,\infty) \times [t, t+1]  \right] < e^{-c(S + \sigma t/2)}\end{align*}
for some $c > 0$, by a comparison with a Poisson random variable (describing the number of arrivals in a certain space-time region; we omit the details). Together with \eqref{eq:bound_before_intervals}, this shows that, if $S$ is large enough,
\begin{equation}\label{eq:bound_after_intervals}
\P\left[r(\xi_t) > \frac{2S}{3} + \sigma'' t \text{ for some } t \geq 0\right] < \varepsilon/2. 
\end{equation}
By Lemma \ref{lem:desc_bar_sides} and \eqref{eq:colorblind}, increasing $S$ if necessary we have
\begin{equation}\label{eq:after_colorblind}
\P\left[\xi_t \equiv 0 \text{ on } \left[\frac{2S}{3} + \sigma''t,\; S + \sigma t\right] \text{ for some } t \geq 0 \right] < \varepsilon/2.
\end{equation}
To conclude,
\begin{align*}
&\P\left[r(\xi_t),\ell(\xi_t) \leq S + \sigma t \text{ for all } t \geq 0\right]\\&\geq \P\left[\text{for all } t \geq 0,\;r(\xi_t) < \frac{2S}{3} + \sigma''t \text{ and } \xi_t(x) \neq 0 \text{ for some } x \in \left[\frac{2S}{3}+\sigma''t, S + \sigma t\right] \right] > 1-\varepsilon. 
\end{align*}

\end{proof}

Given a Harris system $H$ and $s \geq 0$, we define the regenerated interface process $(i^s_t)_{t \geq s}$ as follows:
\begin{equation} \label{eq:def_ist}\text{for any } x \in \Z \text{ and } t \geq s,\;\text{on } \{\lfloor i_s(H) \rfloor = x\},\text{ let } i^{s}_t(H) = x + i_{t-s}(H \circ \theta(x,s)).\end{equation}
Let us explain this definition in words. Using the Harris system $H$, we construct the contact process started from the heaviside configuration $\xi^h_0$ and evolve it up to time $s$, obtaining the configuration $\xi^{h}_s$ with corresponding interface position $i_s$. Then, we artificially put 1's on $(-\infty, \lfloor i_s\rfloor]$ and 2's on $(\lfloor i_s \rfloor,\infty)$ and, using $H_{[s,t]}$, we continue evolving the process; the resulting interface position at time $t$ is $i^s_t$. Note in particular that
\begin{equation} \label{eq:prop_new_inter}
(i^s_{s + t} - i^s_s)_{t \geq 0} \stackrel{\text{(dist.)}}{=} (i_t)_{t\geq 0} \quad \text{ and } \quad |i^s_s - i_s| = 0 \text{ or } \frac12.
\end{equation}
In Section \ref{ss:proof_reg}, we will prove:
\begin{theorem}
\label{thm:interface_regeneration}
For any $\varepsilon > 0$ there exists $K > 0$ such that, for any $s \geq 0$,
\begin{equation}\label{eq:thminterface}\P\left[|i^s_t - i_t| < K \text{ for all } t \geq s\right] > 1-\varepsilon.\end{equation}
\end{theorem}
As a consequence we obtain
\begin{corollary}
\label{cor:interface_motion_tight}
For any $\varepsilon > 0$ and $r > 0$ there exists $K > 0$ such that
\begin{equation*}
\P\left[ \sup_{s \leq t \leq s+r} |i_t - i_s| > K\right] < \varepsilon \quad \text{ for all } s \ge 0.
\end{equation*}
\end{corollary}
\begin{proof} For any $s, r, K$, by \eqref{eq:prop_new_inter},
\begin{align*}
\P\left[\sup_{t \in [s,s+r]} |i_t-i_s| > K \right] \leq \P\left[\sup_{t \in [s,\infty)} |i^s_t - i_t| > K/2\right] + \P \left[ \sup_{t \in [0,r]} |i_t| > (K-1)/2\right].
\end{align*}
Now, for fixed $r$, the second term vanishes as $K \to \infty$, and the first term does so as well by Theorem \ref{thm:interface_regeneration}.
\end{proof}

\section{Convergence of finite-dimensional projections}
\label{s:fdd}
\begin{lemma}\label{lem:interface_nowhere} For any $\varepsilon > 0$ there exists $t_0 > 0$ such that
\begin{equation} \label{eq:interface_nowhere}\P\left[i_t = x \right] < \varepsilon \text{ for any } x \in \frac{1}{2}\Z,\; t \geq t_0.\end{equation}
\end{lemma}
\begin{proof}
Let $\varepsilon > 0$. By \eqref{eq:interface_tightness}, we can obtain $L > 0$ such that for all $t$, $\P[|r_t - \ell_t| > L] < \varepsilon /2$. For any $t$ and $x$ we have
$$\P[i_t = x] \leq \P[|r_t - \ell_t| > L] + \P \left[ \text{there exist } z,w \in [x-L,\;x+L]: \xi^h_t(z) = 1,\;\xi^h_t(w) = 2\right].$$
Switching to the dual process, the second probability can be written as
$$\begin{aligned}&\P\left[\text{there exist } z,w \in [x-L,\;x+L]: \eta^z_t,\;\eta^w_t \neq \triangle,\; \eta^z_t \leq 0,\; \eta^w_t > 0\right]\\[.2cm]
&\hspace{4cm}\leq \sum_{z,w \in [x-L,\;x+L]} \; \P\left[\eta^z_t, \eta^w_t \neq \triangle,\; \eta^z_t \neq \eta^w_t\right].\end{aligned}$$
By \eqref{eq:pairs_meet}, then $t$ is large enough the sum is smaller than $\varepsilon/2$ for any $x$, so we are done.
\end{proof}

\begin{lemma}\label{lem:coinc_eta_inter} For any $\varepsilon > 0$ there exists $t_0 > 0$ such that
\begin{equation}\left| \P[i_t > x\sqrt{t}] - \P\left[\xi^h_t(\lfloor x\sqrt t \rfloor) = 1\;|\; \xi^h_t(\lfloor x\sqrt{t}\rfloor) \neq 0\right]\right| < \varepsilon \text{ for any } x \in \R\text{ and } t \geq t_0.
\end{equation}
\end{lemma}
\begin{proof}
Fix $\varepsilon > 0$. Using \eqref{eq:behTT}, we can choose $S > 0$ such that
\begin{equation}
\label{eq:recap_die_fast}\P\left[S < T^0 < \infty\right] < \varepsilon
\end{equation}
Using Corollary \ref{cor:interface_motion_tight}, we then choose $S'>0$ such that
\begin{equation}\label{eq:recap_move_fast}\P\left[\text{there exists } s \in [t, t+S]:\;|i_s - i_t| \geq S'\right] < \varepsilon \text{ for all } t \geq 0.\end{equation}
Increasing $S'$ if necessary, by  \eqref{eq:interface_tightness} we can also assume that
\begin{equation}
\label{eq:recap_interface}\P\left[|r_t - \ell_t|>S'\right] < \varepsilon \text{ for all } t \geq 0;\\[.2cm]
\end{equation}
Finally, using Lemma \ref{lem:interface_nowhere}, we can choose $t_0 > S$ such that
\begin{equation}\label{eq:recap_nowhere}
\P\left[i_t \in [x-S',\;x+S']\right] < \varepsilon \quad \text{ for all } t \geq 0 \text{ and } x \in \Z.
\end{equation}

Now fix $t \geq t_0$ and $x \in \R$. Denoting by $\triangle$ the symmetric difference between sets, we have the following estimates:
\begin{align}\label{eq:rcp_bound_1}
\begin{split} &\P\left[\{\xi^h_t(\lfloor x\sqrt{t}\rfloor) = 1 \} \;\triangle\; \{\xi^h_t(\lfloor x \sqrt{t} \rfloor) \neq 0,\; i_t > x\sqrt{t} \} \right] \\&\hspace{3cm}\leq \P\left[|i_t - x\sqrt{t}| \leq S' \right] + \P\left[|r_t - \ell_t| > S'\right] \stackrel{\eqref{eq:recap_interface},\eqref{eq:recap_nowhere}}{\leq} 2\varepsilon; \end{split}\\[.2cm]
\label{eq:rcp_bound_2}\begin{split} &\P\left[\{i_{t-S} > x\sqrt{t}\}\;\triangle\;\{i_t > x\sqrt{t}\} \right] \\&\qquad\leq \P\left[|i_t - x\sqrt{t}| \leq S'\right] + \P\left[|i_{t-S} - x\sqrt{t}| \leq S'\right] + \P\left[|i_t - i_{t-S}| > S'\right] \stackrel{\eqref{eq:recap_move_fast}, \eqref{eq:recap_nowhere}}{<} 3\varepsilon; \end{split}\\[.2cm]
\begin{split}&\P\left[\Z\times \{t-S\} \;\leftrightarrow \; (x\sqrt{t},t),\;\Z\times\{0\} \;\nleftrightarrow\;(x\sqrt{t},t)\right]\\&\hspace{4.5cm} \leq \P\left[(0,0)\;\leftrightarrow \Z\times \{S\},\; (0,0)\;\nleftrightarrow \;\infty \right] \stackrel{\eqref{eq:recap_die_fast}}{<} \varepsilon.\end{split}\label{eq:rcp_bound_3}
\end{align}
With these bounds at hand, we are ready to prove the statement of the lemma. In the following computation, the symbol $\approx$ means that the absolute value of the difference between the left-hand side and the right-hand side is at most $5\varepsilon$. 
\begin{align*}
\P\left[\xi^h_t(\lfloor x\sqrt{t}\rfloor) = 1\right] &\stackrel{\eqref{eq:rcp_bound_1}}{\approx} \P\left[\xi^h_t(\lfloor x\sqrt{t}\rfloor) \neq 0,\; i_t>x\sqrt{t}\right]\\[.2cm]
&\stackrel{\eqref{eq:rcp_bound_2},\eqref{eq:rcp_bound_3}}{\approx} \P\left[\Z \times \{t-S\} \;\leftrightarrow \; (\lfloor x\sqrt{t}\rfloor, t),\; i_{t-S} > x \sqrt{t}\right]\\[.2cm]
&= \P\left[\Z \times \{t-S\} \;\leftrightarrow \; (\lfloor x\sqrt{t} \rfloor, t)\right] \cdot \P\left[i_{t-S} > x\sqrt{t}\right]\\[.2cm]
&\stackrel{\eqref{eq:rcp_bound_3}}{\approx} \P\left[\xi^h_t(\lfloor x\sqrt{t}\rfloor) \neq 0\right] \cdot \P\left[i_t > x\sqrt{t}\right].
\end{align*}
We then have
\begin{align*}
&\left| \P\left[i_t > x\sqrt{t}\right] -\P\left[ \xi^h_t(\lfloor x\sqrt{t}\rfloor) = 1\;|\; \xi^h_t(\lfloor x\sqrt{t}\rfloor)\neq 0\right]\right|\\[.2cm]
&=\P\left[\xi^h_t(\lfloor x\sqrt{t}\rfloor) \neq 0\right]^{-1}\cdot \left|\P\left[ \xi^h_t(\lfloor x\sqrt{t}\rfloor) \neq 0\right]\cdot \P\left[i_t > x\sqrt{t}\right] - \P\left[\xi^h_t(\lfloor x \sqrt{t}\rfloor) = 1\right] \right|\\[.2cm]
&\leq \P[T^0 = \infty]^{-1} \cdot 15\varepsilon,
\end{align*}
that is, at most a universal constant times $\varepsilon$. This completes the proof.
\end{proof}

\begin{proposition}
\label{prop:1dd}
As $t \to \infty$, $\frac{i_t}{\sqrt{t}}$ converges in distribution to $\mathcal{N}(0,\sigma^2)$, where $\sigma$ is as in \eqref{eq:conv_norm}.
\end{proposition}
\begin{proof}
The statement follows from \eqref{eq:conv_norm}, Lemma \ref{lem:coinc_eta_inter} and the fact that
$$\P\left[\xi^h_t(\lfloor x\sqrt{t}\rfloor) = 1\;|\;\xi^h_t(\lfloor x\sqrt{t}\rfloor) \neq 0\right] = \tilde \P^{\lfloor x \sqrt{t}\rfloor}\left[ \eta^{\lfloor x \sqrt{t} \rfloor}_t \leq 0\right] = \tilde \P^0\left[\frac{\eta^0_t}{\sqrt{t}} \leq - \frac{\lfloor x \sqrt{t} \rfloor}{\sqrt{t}}\right].$$
\end{proof}

\begin{proofof}\textit{Proposition \ref{prop:fdd}.}
Fix $0 < a_1 < \ldots < a_k$. We have
\begin{align} \nonumber&t^{-1/2}\left(i_{a_1 t},\; i_{a_2 t} - i_{a_1 t},\;\ldots,\; i_{a_k t} - i_{a_{k-1}t}\right)
\\\label{eq:term_g1}&\qquad= t^{-1/2}\left(i_{a_1 t},\; i^{a_1 t}_{a_2 t} - i^{a_1 t}_{a_1 t},\;\ldots,\; i^{a_{k-1}t}_{a_k t} - i^{a_{k-1}t}_{a_{k-1}t}\right)\\
\label{eq:term_g2}&\qquad\quad+t^{-1/2} \left(0,\;i_{a_2 t} - i^{a_1t}_{a_2t}  + i^{a_1t}_{a_1t} -  i_{a_1t} ,\;\ldots,\;i_{a_kt} - i^{a_{k-1}t}_{a_kt} + i^{a_{k-1}t}_{a_{k-1}t} - i_{a_{k-1}t}\right).
\end{align}
Theorem \ref{thm:interface_regeneration} implies that the term in \eqref{eq:term_g2} converges to zero in probability as $t \to \infty$. Since the elements of the vector in \eqref{eq:term_g1} are independent and satisfy
$$i^{a_i t}_{a_{i+1} t} - i^{a_i t}_{a_i t} \stackrel{\text{(dist.)}}{=} i_{(a_{i+1} - a_i) t},$$
Proposition \ref{prop:1dd} shows that \eqref{eq:term_g1} converges in distribution, as $t \to \infty$, to the distribution prescribed in Proposition \ref{prop:fdd}.
\end{proofof}

\section{Tightness in $D$}
Most of the effort in this section will go into proving the following uniform bound on the displacement of the interface position.
\begin{lemma}\label{lem:inter_not_move}
For any $\varepsilon > 0$ there exists $U > 0$ such that, for large enough $t$ and any $r > 0$,
\begin{equation*}
\P\left[\sup_{r \leq s \leq r+t} |i_s - i_r| \leq U\sqrt{t} \right] > 1 - \frac{\varepsilon}{U^2}.
\end{equation*}
\end{lemma}
The proof will depend on several preliminary results. Before turning to them, let us first explain how Lemma \ref{lem:inter_not_move} allows us to conclude.\\[.2cm]
\begin{proofof}{\textit{Proposition \ref{prop:tight}.}}
For each $t > 0$, define the process $X^{(t)}$ by
$$X^{(t)}_s = t^{-1/2}\cdot i_{st},\qquad s\geq 0.$$
We want to show that the family of processes  $\{X^{(t)}:t \geq 0\}$ is tight in $D[0,\infty)$. As explained in  Section 16 of \cite{bill}, it is sufficient to prove that, for every $m > 0$,
\begin{equation}\begin{split}
\text{for all }m > 0 \text{ and } \varepsilon > 0 \text{ there exists }\delta > 0 \text{ and } t_0 > 0 \text{ such that }\\
\P\left[\sup_{k \in\{0,\ldots,\lfloor m/\delta \rfloor\}}\;\sup_{s \in [k\delta,\;(k+1)\delta ]} |X^{(t)}_s - X^{(t)}_{k\delta}| > \varepsilon \right] < \varepsilon \text{ for all }  t \geq t_0.\end{split}
\end{equation}
By the identity $X^{(t)}_m = \sqrt{m} \cdot X^{(mt)}_1$, it is sufficient to treat $m=1$. Then the above condition becomes
\begin{equation}
\label{eq:bill_cond}
\begin{split}&\text{for all } \varepsilon > 0 \text{ there exists } \delta > 0 \text{ such that, for large enough } t,\\&\qquad\P\left[\sup_{k \in\{0,\ldots,\lfloor 1/\delta \rfloor\}}\;\sup_{s \in [k\delta t,\;(k+1)\delta t]} |i_s - i_{k\delta t}| > \varepsilon \sqrt{t} \right] < \varepsilon.\end{split}
\end{equation}
Given $\varepsilon > 0$, using Lemma \ref{lem:inter_not_move}, we can find $U > 0$ and $t_0> 0$ such that, if $t \ge t_0$,
\begin{equation}\label{eq:just_bill}
\sup_{r\geq0}\;\P\left[\sup_{r \leq s \leq r + t}|i_s - i_{r}| > U\sqrt{t}\right] < \frac{\varepsilon^3}{U^2}
\end{equation}
Now, set $\delta = (\varepsilon/U)^2$. We then have, for $t \ge t_0/\delta$ and $t' = \delta t$,
\begin{align*}
&\P\left[\sup_{k \in\{0,\ldots,\lfloor 1/\delta \rfloor\}}\;\sup_{s \in [k\delta t,\;(k+1)\delta t]} |i_s - i_{k\delta t}| > \varepsilon \sqrt{t} \right]\leq \frac{1}{\delta} \cdot \sup_{r\geq 0}\;\P\left[\sup_{s \in [r,r+\delta t]}|i_s - i_r| > \varepsilon\sqrt{t} \right]\\[.2cm]
&\quad= \frac{1}{\delta} \cdot \sup_{r\geq 0}\;\P\left[\sup_{s \in [r,r+t']}|i_s - i_r| > \varepsilon\sqrt{\frac{t'}{\delta}} \right] = \frac{U^2}{\varepsilon^2}\cdot \sup_{r\geq 0}\;\P\left[\sup_{s \in [r,r+t']}|i_s - i_r| > U\sqrt{t'} \right] \stackrel{\eqref{eq:just_bill}}{<} \varepsilon
\end{align*}
as required in \eqref{eq:bill_cond}.
\end{proofof}

Our first step towards the proof of Lemma \ref{lem:inter_not_move} are the following generalizations of Lemmas \ref{lem:no_renewals_ab} and \ref{lem_x_goes_to_y_first}. Since the proofs are line-by-line repetitions of the proofs of these earlier results, we omit them. 

\begin{lemma}\label{lem_x_goes_to_y}
\begin{enumerate}
\item 
There exists $c > 0$ such that the following holds. For any $(x_1,r_1), (x_2,r_2) \in \Z \times [0,\infty)$ and $b > a \geq \max(r_1, r_2)$, we have
\begin{equation}\label{eq:times_after}\P\left[\begin{array}{l}\eta^{(x_1,r_1)}_s, \eta^{(x_2,r_2)}_s \neq \triangle \text{ for all } s \in [a,b],\\[.2cm] \min\big( T^{\big(\eta^{(x_1,r_1)}_s, s\big)},\; T^{\big(\eta^{(x_2,r_2)}_s, s\big)}\big)< \infty \text{ for all } s \in [a, b] \end{array}\right] < e^{-c(b-a)}.\end{equation}
\item
Let $(x_1, r_1), (x_2, r_2)\in \Z\times [0,\infty)$ and $\sigma$ be a stopping time (with respect to the sigma-algebra of Harris systems) with $\sigma \geq \max( r_1,r_2)$ almost surely. Let
$$\uptau = \inf\{t \geq \sigma: \eta^{(x_1,r_1)}_t, \eta^{(x_2,r_2)}_t \neq \triangle \text{ and }T^{\big(\eta^{(x_1,r_1)}_t,t\big)} = T^{\big(\eta^{(x_2,r_2)}_t,t\big)} = \infty\}. $$
For any $y_1, y_2 \in \Z$ and events $E, F$ on Harris systems,
\begin{equation}\begin{split}&\P\left[\uptau < \infty,\;H_{[0,\uptau]} \in E,\;\eta^{(x_1,r_1)}_\uptau = y_1,\;\eta^{(x_2,r_2)}_\uptau = y_2\text{ and } H \circ \theta(0,\uptau) \in F \right] \\[.2cm]
&=  \P\left[\uptau < \infty,\; H_{[0,\uptau]} \in E,\;\eta^{(x_1,r_1)}_\uptau = y_1,\;\eta^{(x_2,r_2)}_\uptau = y_2\right]
\cdot \tilde \P^{\{y_1,y_2\}} \left[H \in F\right].
\end{split} \label{eq:x_goes_to_y}\end{equation}
\end{enumerate}
\end{lemma}

Given $u > 0$ and $t > 0$, define $K^{u, t} = -13u\sqrt{t}$ and the intervals
$$\begin{aligned}
&I^{u, t}_k = K^{u, t} + \left((4k-1)u\sqrt{t}, \; 4ku\sqrt{t} \right), \qquad 1 \leq k \leq 3;\\
&J^{u, t}_0 = (-\infty, K^{u, t}];\\
&J^{u, t}_k = K^{u, t} + \left[(4k-3)u\sqrt{t},\; (4k-2)u\sqrt{t} \right], \qquad 1 \leq k \leq 3;\\
&J^{u, t}_4 = K^{u, t} + [13u\sqrt{t},\; \infty) = [0, \infty).
\end{aligned}$$
We will often omit the superscripts and write $K$, $I_k$ and $J_k$. These definitions, as well as the event treated in the following lemma, are illustrated in Figure \ref{fig4brownie}.

\begin{figure}[htb]
\begin{center}
\setlength\fboxsep{0pt}
\setlength\fboxrule{0.5pt}
\fbox{\includegraphics[width = 1.0\textwidth]{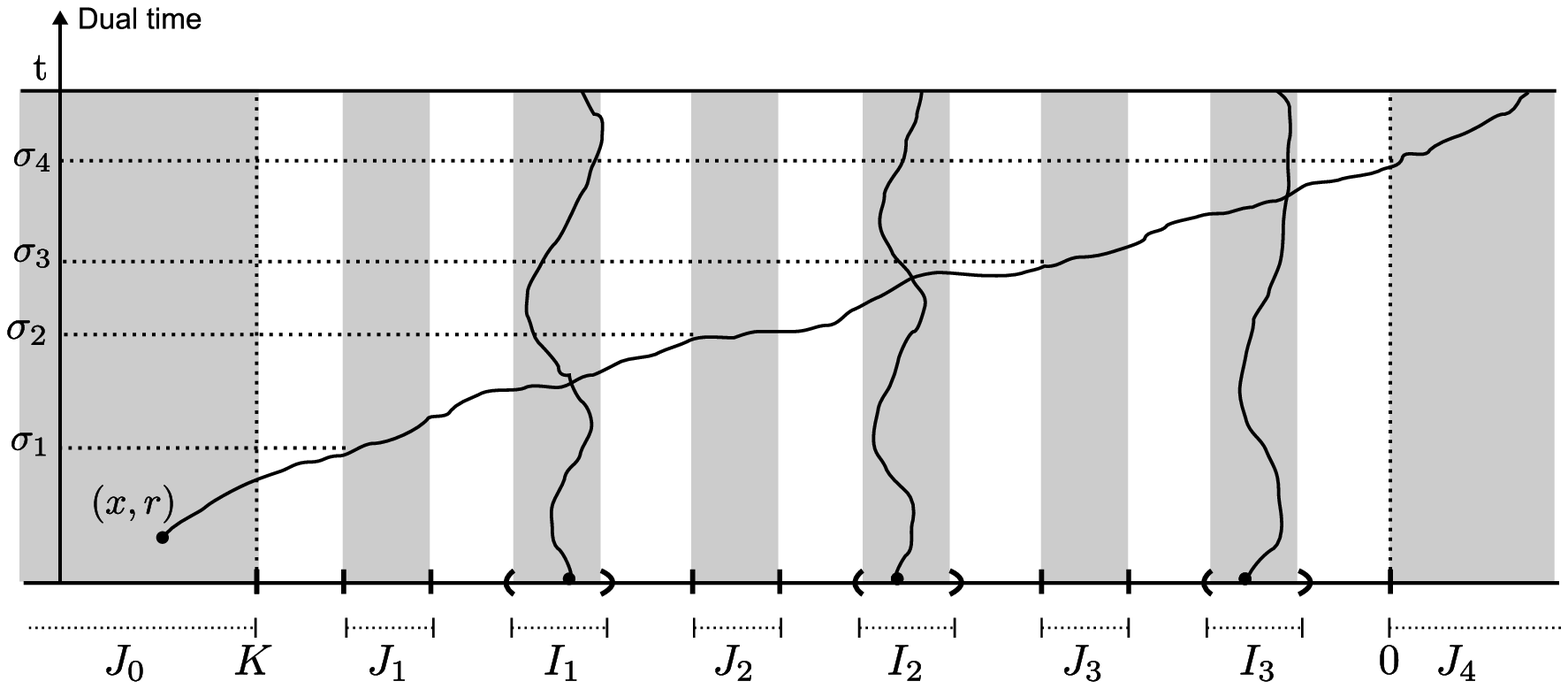}}
\end{center}
\caption{The event of Lemma \ref{cla4onevstun}.}
\label{fig4brownie}
\end{figure}

\begin{lemma}
\label{cla4onevstun}
For any $u > 0$, there exists $C = C(u) > 0$ and $t_0 = t_0(u) > 0$ such that, for any $t \ge t_0$, $(x, r) \in J^{u, t}_0 \times [0, t]$, $x_1 \in I_1^{u,t}$, $x_2 \in I_2^{u,t}$ and $x_3 \in I_3^{u,t}$, we have
$$\P\left[\eta^{(x, r)}_t \geq 0 \text{ and } \eta^{x_i}_s \in I_k^{u, t} \text{ for all } k \in \{1,2,3\} \text{ and } s \in [0,t]  \right] \leq \frac{C}{t^{3/2}}.$$
\end{lemma}
\begin{proof}
Fix $u > 0$. Choose $t$ large enough that \begin{equation}\label{eq:choice_of_t0}2 t^{1/4} < u\sqrt{t}.\end{equation} Fix $(x, r) \in J_0 \times [0, t],\; x_1 \in I_1,\; x_2 \in I_2,\; x_3 \in I_3$. Define
$$\sigma_k = \inf\{s: \eta^{(x,r)}_s \geq \inf J_k\},\quad 1 \leq k \leq 4.$$
The probability in the statement of the lemma is less than 
$$\P[\sigma_4 \leq t,\; \text{for all } s \leq \sigma_4 \text{ and }k \leq 3,\; \eta^{x_k}_s \in I_k].$$ We will show that
\begin{equation}\label{eqn4showthat}\begin{split} &\P\left[\sigma_4 \leq t,\; \eta^{x_k}_s \in I_k \text{ for all } s \le \sigma_4 \text{ and } k \leq 3\right] \\[.2cm]&\hspace{3.5cm}\leq Ce^{-ct^\gamma} + \frac{C}{\sqrt{t}}\; \P\left[\sigma_3 \leq t,  \eta^{x_k}_s \in I_k \text{ for all } k\leq 2 \text{ and } s \leq \sigma_3\right]\end{split}\end{equation}
for some $c, C, \gamma$ that only depend on $u$. To this end, we first define the events
\begin{align*}
&E_1 = \left\{\exists s, s'\in[r,t]: |s - s'| < t^{1/8},\;\eta^{(x,r)}_s\neq \triangle,\; \eta^{(x,r)}_{s'} \neq \triangle,\;|\eta^{(x,r)}_s - \eta^{(x,r)}_{s'}| > t^{1/4} \right\},\\[.3cm]
&E_2 = \left\{\begin{array}{l}\exists s, s' \in [r,t]: |s-s'| \geq t^{1/8},\;\eta^{(x,r)}_u\neq \triangle,\;\eta^{x_3}_u \neq \triangle\\[.2cm] \text{ and } \min\left(T^{(\eta^{(x,r)}_u,u)},\;T^{(\eta^{x_3}_u,u)}\right) < \infty \text{ for all } u \in [s,s']\end{array}\right\}
\end{align*}
and the random time
\begin{equation*}\uptau = \inf\{t \geq \sigma_3: \eta^{(x,r)}_t\neq \triangle,\;\eta^{x_3}_t \neq \triangle,\; (\eta^{(x,r)}_t,\;t) \leftrightarrow \infty,\;\left(\eta^{x_3}_t,t\right)\leftrightarrow \infty \}.\end{equation*} 
By \eqref{eq:bondRW} and \eqref{eq:x_goes_to_y}, we can find $\upgamma > 0$ such that
\begin{equation}
\P[E_1 \cup E_2] \leq Ce^{-ct^\upgamma}.\label{eq:first_bound_comp}
\end{equation}
Using \eqref{eq:choice_of_t0}, we have
$$\begin{aligned}&\left\{\sigma_4 \leq t;\;\eta^{x_k}_s \in I_k\text{ for all } s \leq \sigma_4,\; k \leq 3\right\} \cap E_1^c \\[.3cm]&\hspace{1cm}\subseteq \left\{ \begin{array}{l} \sigma_3 \leq t - t^{1/8};\; \eta^{x_k}_s \in I_k \text{ for all } s \leq \sigma_3 \text{ and } k \leq 2;\\[.2cm] \eta^{(x,r)}_s \leq \sup J_3 \text{ and } \eta^{x_3}_s \in I_3 \text{ for all } s \in [\sigma_3,\;\sigma_3 + t^{1/8}];\\[.2cm] \eta^{(x,r)}_s > \eta^{x_3}_s + u\sqrt{t} \text{ for some } s \in [\sigma_3 + t^{1/8}, t] \end{array}\right\}.\end{aligned}$$
We thus also get
$$\begin{aligned}
&\left\{\sigma_4 \leq t,\;\eta^{x_k}_s \in I_k\text{ for all } s \leq \sigma_4 \text{ and }\; k \leq 3\right\} \cap E_1^c \cap E_2^c
\\[.3cm]&\subseteq\left\{ \begin{array}{l}\sigma_3 \leq t - t^{1/8},\;\eta^{x_k}_s \in I_k \text{ for all }s \leq \sigma_3 \text{ and } k \leq 2,\\[.2cm]\uptau \leq \sigma_3 + t^{1/8},\; \eta^{(x,r)}_{\uptau} \leq \sup J_3,\;\eta^{x_3}_{\uptau} \in I_3,\; \eta^{(x,r)}_s > \eta^{x_3}_s + u\sqrt{t} \text{ for some } s \in[\uptau,t]\end{array} \right\} \\[.3cm]
&\subseteq \left\{ \begin{array}{l}\sigma_3 \leq t,\;\eta^{x_k}_s \in I_k \text{ for all } s \leq \sigma_3 \text{ and } k \leq 2,\\[.2cm]\uptau < \infty,\;\eta^{(x,r)}_{\uptau} \leq \sup J_3,\;\eta^{x_3}_{\uptau} \in I_3,\; \eta^{(x,r)}_s > \eta^{x_3}_s + u\sqrt{t} \text{ for some } s \in[\uptau,\uptau + t]\end{array} \right\}\\[.3cm]
&= \bigcup_{z\leq \sup J_3} \;\bigcup_{w \in I_3} \left\{ \begin{array}{l}\sigma_3 \leq t,\;\eta^{x_k}_s \in I_k \text{ for all } s\leq \sigma_3 \text{ and } k \leq 2,\\[.2cm]\uptau  < \infty,\;\eta^{(x,r)}_{\uptau} = z,\;\eta^{x_3}_{\uptau} = w,\; \eta^{(z,\uptau)}_s > \eta^{(w,\uptau)}_s + u\sqrt{t} \text{ for some } s \in [\uptau, \uptau + t] \end{array} \right\}.
\end{aligned}$$
Using this set inclusion and \eqref{eq:x_goes_to_y} we obtain
$$\begin{aligned}
&\P\left[ \left\{\sigma_4 \leq t,\;\eta^{x_k}_s \in I_k \text{ for all } s \leq \sigma_4 ,\; k\leq 2 \right\} \cap E_1^c \cap E_2^c \right]\\[.3cm]
&\leq \sum_{z \leq \sup J_3}\; \sum_{w \in I_3} \P\left[\sigma_3 \leq t,\;\eta^{x_k}_s \in I_k \text{ for all } s \leq \sigma_3 \text{ and } k \leq 2,\; \uptau < \infty,\;\eta^{(x,t)}_{\uptau} = z,\; \eta^{x_3}_{\uptau}= w \right]\\[.1cm]&\hspace{6cm}\cdot  \tilde \P^{\{z,w\}}\left[\eta^z_s > \eta^w_s +u\sqrt{t} \text{ for some } s \in [0,t]\right] \\[.3cm]
&\leq \P\left[\sigma_3 \leq t,\;\eta^{x_k}_s \in I_k \text{ for all } s \leq \sigma_3 \text{ and } k \leq 2 \right]\\[.1cm] &\hspace{6cm}\cdot  \sup_{z \leq \sup J_3,\; w \in I_3}   \tilde \P^{\{z,w\}}\left[\eta^z_s > \eta^w_s + u\sqrt{t} \text{ for some } s \in [0,t]\right]  \\[.3cm]
&\stackrel{\eqref{eq:key_rw_estimate_rs}}{\leq} \frac{C}{\sqrt{t}}  \cdot \P\left[\sigma_3 \leq t,\;\eta^{x_k}_s \in I_k \text{ for all } s \leq \sigma_3 \text{ and } k \leq 2 \right].
\end{aligned}$$
The desired result now follows from iterating this computation.
\end{proof}

\begin{lemma}\label{lem:still_dual}
For any $\varepsilon > 0$ there exists $U > 0$ such that, for large enough $t$,
\begin{equation}\label{eq:still_dual}
\P\left[\text{there exists } (x,r) \in [-U\sqrt{t} - t^{1/4}, -U\sqrt{t}] \times \{0,1,\ldots, \lfloor t \rfloor\}:\;\eta^{(x,r)}_t > 0 \right]<\frac{\varepsilon}{U^2}.
\end{equation}
\end{lemma}
\begin{proof}
Let $N$ be a large integer to be chosen later. Define
$$u = N^2,\qquad U = 13 u,\qquad t> t_0(u),\text{ where } t_0(u)\text{ is as in Lemma }\ref{cla4onevstun}.$$
For $k \in \{1,2,3\}$, let $\bar I_k^{u,t}$ be the set of $N$ points (or $N+1$ points, depending on parity) that are closest to the middle point of $I_k^{u,t}$. Define the events
\begin{align*}
&B_1 = \left\{ T^{\bar I_k^{u,t}} < \infty \text{ for some } k \in \{1,2,3\}\right\},\\[.2cm]
&B_2 = \left\{\text{there exist } k \in \{1,2,3\},\; y \in \bar I^{u,t}_k,\; s\leq t \text{ such that } \eta^y_s \neq \triangle, \;\eta^y_s \notin I^{u,t}_k\right\},\\[.2cm]
&B_3 = \left\{\begin{array}{c}\text{there exist } (x,r) \in [-U\sqrt{t}-t^{1/4}, -U\sqrt{t}] \times \{0,\ldots, \lfloor t \rfloor \},\\[.15cm] y_1 \in \bar I^{u,t}_1,\;y_2 \in \bar I^{u,t}_2,\;y_3 \in \bar I^{u,t}_3 \text{ such that }\\[.15cm]
\eta^{(x,r)}_t > 0,\;\eta^{y_i}_s \in I^{u,t}_i \text{ for each } i =1, 2, 3,\;s\leq t.\end{array}\right\}.
\end{align*}
We then have
$$\left\{ \text{there exists } (x,r) \in [-U\sqrt{t} - t^{1/4},-U\sqrt{t}]\times \{0,1,\ldots, \lfloor t\rfloor\}: \eta^{(x,r)}_t > 0 \right\} \subseteq \cup_{k=1}^3 B_k.$$

In what follows, $c$ and $C$ will denote constants that only depend on $\lambda$ and $R$, and $C_N$ will denote constants that also depend on $N$. Of course, since $u = N^2$, constants that depend on $u$ also depend on $N$. Equation \eqref{eq:behTT} implies that, for some $c>0$,
$$\P[B_1] \leq 3e^{-cN}.$$

To bound the probability of $B_2$, fix $k \in \{1,2,3\}$ and $y \in \bar I^{u,t}_k$. As long as $t$ is large enough that $N < u\sqrt{t}$, we have
$$\{\text{there exists } s \leq t:\eta^y_s \notin I^{u,t}_k\} \subseteq \left\{\sup_{0\leq s \leq t} |\eta^y_s - y| > u\sqrt{t}/2\right\}$$
and, by \eqref{eq:bondRW}, the probability of this event is less than $Ce^{-c\frac{u^2t}{4t}} + Cte^{-c\frac{u\sqrt{t}}{2}}$ for some $c, C > 0$. We thus get
$$\P[B_2] \leq 3N\left(Ce^{-c\frac{u^2t}{4t}} + Cte^{-c\frac{u\sqrt{t}}{2}}\right) = 3N\left(Ce^{-c\frac{N^4}{4}} + Cte^{-c\frac{N^2\sqrt{t}}{2}}\right).$$ 

We now turn to $B_3$. Note that there are at most $t^{5/4}$ candidates for $(x,r)$ and $N^3$ candidates for $y_1,y_2,y_3$. Using Lemma \ref{cla4onevstun}, there exists $C_N > 0$ such that 
$$\P[B_3] \leq N^3 \cdot t^{5/4} \cdot \frac{C_N}{t^{3/2}} \leq \frac{C_N\cdot N^3}{t^{1/4}}.$$

Putting these bounds together and rearranging constants, we get
\begin{align}
\nonumber &U^2 \cdot \P\left[\text{there exists } (x,r) \in [-U\sqrt{t} - t^{1/4},-U\sqrt{t}]\times \{0,1,\ldots,\lfloor t \rfloor \}: \eta^{(x,r)}_t > 0 \right]\\[.2cm]
 &\qquad\qquad\qquad\qquad\leq CN^4e^{-cN} + CN^5 e^{-cN^4} + CtN^5 e^{-cN^2\sqrt{t}} + \frac{C_NN^7}{t^\frac{1}{4}}.\label{eq:long_bound}
\end{align}
Now, given $\varepsilon > 0$, we first choose $N^*$ such that the sum of the first two terms on \eqref{eq:long_bound} is less than $\varepsilon/2$. Next, we choose $t^* > t_0((N^*)^2)$ such that, for $N^*$ and any $t > t^*$, the sum of the third and fourth terms in \eqref{eq:long_bound} is less than $\varepsilon /2$. This completes the proof of the lemma, with $U = 13(N^*)^2$.
\end{proof}

\begin{lemma}\label{lem:rell_not_move}
For any $\varepsilon > 0$ there exists $U > 0$ such that, for large enough $t$,
$$\P\left[|r_s|,|\ell_s| \leq U\sqrt{t} \text{ for all } s\leq t\right] > 1 - \frac{\varepsilon}{U^2}.$$
\end{lemma}
\begin{proof}
Given $\varepsilon > 0$, we will find $U > 0$ such that
\begin{equation}\P\left[\inf_{0\leq s \leq t} \ell_s < -U\sqrt{t} - t^{1/4} \right] < \frac{\varepsilon}{U^2}  \text{ if $t$ is large enough} \label{eq:U_2U}\end{equation}
and
\begin{equation}
\label{eq:U_2U2}
\P\left[\inf_{0 \leq s \leq t}\ell_s \geq - U\sqrt{t} - t^{1/4},\; \inf_{0\leq s \leq t} r_s < -2U\sqrt{t}  \right] \xrightarrow{t \to \infty} 0;
\end{equation}
the statement of the lemma clearly follows from these statements and symmetry.

For \eqref{eq:U_2U}, we remark that, using the joint construction of the multitype contact process and the ancestor processes,  \eqref{eq:still_dual} can be rewritten as
\begin{equation}
\P\left[\text{there exists } (x, r) \in [-U\sqrt{t} - t^{1/4}, -U\sqrt{t}] \times \{0,1,\ldots, \lfloor t \rfloor\}:\xi^h_r(x) = 2 \right] < \frac{\varepsilon}{U^2}.
\end{equation}
Letting $A_t$ be the event that appears in the above probability, we also have
\begin{align*}
&\P\left[A_t^c \cap \left\{\inf_{0\leq s\leq t} \ell_s < -U\sqrt{t} - t^{1/4} \right\}\right] \\&\leq \sum_{r = 0}^{\lfloor t \rfloor} \P\left[[-U\sqrt{t},\infty) \times \{r\} \;\leftrightarrow \; (-\infty,-U\sqrt{t} - t^{1/4}]\times [r,r+1]\right] \leq te^{-ct^{1/4}}\end{align*}
by a comparison with a Poisson random variable. \eqref{eq:U_2U} is thus proved.

For \eqref{eq:U_2U2}, note that
$$\begin{aligned}&\P\left[\inf_{0 \leq s \leq t}\ell_s \geq -U\sqrt{t} - t^{1/4},\; \inf_{0\leq s \leq t} r_s < -2U\sqrt{t}  \right] \\[.2cm]&\leq \P \left[\text{there exists } s \in [0,t] \text{ such that } \xi^h_s \equiv 0 \text{ on } [-2U\sqrt{t},\;-U\sqrt{t} - t^{1/4}\right].\end{aligned}$$
The probability on the right-hand side can be bounded above similarly to how we proceeded in Lemma \ref{lem:desc_bar_sides}, so that \eqref{eq:U_2U2} follows.
\end{proof}

\begin{proofof}{\textit{Lemma \ref{lem:inter_not_move}}}
We recall the definition of $i_s^t$ in \eqref{eq:def_ist}. Given $\varepsilon > 0$, using the previous lemma  we choose $U$ so that
$$\P\left[\sup_{0\leq s \leq t} |i_s| \leq \frac{U}{2}\sqrt{t} \right] > 1 - \frac{\varepsilon}{2U^2}. $$ 
Using Theorem  \ref{thm:interface_regeneration}, we choose $K$ so that, for any $r > 0$,
$$\P\left[|i_s - i^r_s| \leq K \text{ for all } s \geq r\right] > 1 - \frac{\varepsilon}{2U^2}.$$
Then, if $t$ is large enough that $U\sqrt{t} > 2K$ we have, for any $r > 0$,
$$\begin{aligned}
\P\left[\sup_{r \leq s \leq r+t} |i_s - i_r| > U\sqrt{t}\right] &\leq \P\left[\sup_{r \leq s \leq r+t} |i_s - i^r_s| >  \frac{U}{2}\sqrt{t} \right] + \P\left[\sup_{r \leq s \leq r+t} |i^r_s - i_r| > \frac{U}{2}\sqrt{t} \right]\\[.2cm]
&\leq  \P\left[\sup_{r \leq s \leq r+t} |i_s - i^r_s| >K \right] + \P\left[\sup_{0 \leq s \leq t} |i_s| > \frac{U}{2}\sqrt{t} \right] < \frac{\varepsilon}{U^2}.
\end{aligned}$$
\end{proofof}

\section{Interface regeneration}\label{ss:proof_reg}
In this section we will prove Theorem \ref{thm:interface_regeneration}. We will often consider multitype contact processes with different initial configurations simultaneously. When we do so, we always assume that all these processes are constructed on the same probability space, using a single Harris system $H$.

We start defining some classes of subsets of the space of configurations $\{0,1,2\}^\Z$. Recall the definition of $\Omega \subset \{0,1,2\}^\Z$ in \eqref{eq:def_of_Omega}. Define 
\begin{align}\label{eq:def_gamma}
\Gamma_{S,L} = \left\{\begin{array}{ll}\xi \in \Omega:& \text{ there exist } a < b \text{ with } b-a \leq L,\;r(\xi),\ell(\xi) \in (a,b),\\[.2cm]& \xi \equiv 1 \text{ on } [a-S,a] \text{ and } \xi \equiv 2 \text{ on } [b,b+S]\end{array}\right\},\; S,L > 0.
\end{align}
The homogeneously and fully occupied intervals $[a-S, a]$, $[b, b+S]$ that appear in the above definition will be referred to as ``isolation segments''. The reason is that we think of them as isolating the interface (which is contained in $(a,b)$) from the ``outside'' $[a-S,b+S]^c$, so that, if $S$ is large, we can hope that the configuration in the outside never has any effect on the evolution of the interface.

Our second class of configurations will depend on a preliminary definition. Given $\xi_0 \in \Omega$, let $$\tilde \xi_0 = \mathds{1}_{(-\infty, \lfloor i(\xi_0)\rfloor]} + 2 \cdot \mathds{1}_{(\lfloor i(\xi_0)\rfloor,\infty)}.$$
Also let $(\xi_t)$ and $(\tilde \xi_t)$ be contact processes started from $\xi_0$ and $\tilde \xi_0$, respectively (constructed with the same Harris system). We now let
\begin{equation}\Omega_{\varepsilon,K} = \left\{\xi_0 \in \Omega: \P\left[|i(\xi_t) - i(\tilde \xi_t)| < K \text{ for all } t\right]> 1-\varepsilon \right\}.\label{eq:def_omega_eps_K}\end{equation}
We will separately prove the following two propositions:
\begin{proposition}\textbf{(Large isolation segments allow for regeneration).}\label{prop:omega_eps_K}
For any $\varepsilon > 0$ there exists $S > 0$ such that the following holds. For any $L > 0$ there exists $K = K(\varepsilon, S, L) >0$ such that $\Gamma_{S,L} \subseteq \Omega_{\varepsilon,K}$.
\end{proposition}

\begin{proposition}\textbf{(Large isolation segments are found not too far).}\label{prop:gamma_sl}
For any $\varepsilon > 0$ and $S > 0$ there exists $L > 0$ such that, for any $t \geq 0$,
$$\mathbb{P}\left[\xi^h_t \in \Gamma_{S,L}\right] > 1-\varepsilon.$$
\end{proposition}

\begin{proofof}{\textit{Theorem} \ref{thm:interface_regeneration}.}
Fix $\varepsilon > 0$. Choose $S = S(\varepsilon)$ as in Proposition \ref{prop:omega_eps_K}, then choose $L = L(\varepsilon, S)$ as in Proposition \ref{prop:gamma_sl}, and finally choose $K = K(\varepsilon,S,L)$ as in Proposition \ref{prop:omega_eps_K}. Now, for any $t \geq 0$ we have
\begin{equation}
\label{eq:inter_reg_rew}\P\left[\xi^h_t \in \Omega_{\varepsilon, K} \right] \stackrel{}{\geq} \P \left[\xi^h_t \in \Gamma_{S,L}\right] > 1-\varepsilon.
\end{equation}
Now, for any $s \geq 0$ we have
\begin{align*}
\P\left[\sup_{t \geq s} |i^s_t - i_t| > K \right] \leq \P\left[\xi^h_s \notin \Omega_{\varepsilon, K} \right] + \P\left[\left.\sup_{t \geq s} |i^s_t - i_t| > K\;\right|\; \xi^h_s \in \Omega_{\varepsilon, K}\right] < 2\varepsilon.
\end{align*}
\end{proofof}

\subsection{Proof of Proposition \ref{prop:omega_eps_K}}

\begin{lemma}
\label{prop:almost_all_couples_all}
For any $\varepsilon > 0$ and $L > 0$ there exists $t_0 > 0$ such that the following holds. If $I$ is an interval of length at most $L$ and $\hat \xi_0,\;\doublehat{\xi}_0 \in \{0,1,2\}^{\Z}$ are such that $\hat \xi_0(x) = \doublehat{\xi}_0(x) \neq 0$ for all $x \in \Z\backslash I$, then
$$\mathbb{P}\left[\hat\xi_t = \doublehat{\xi}_t \text{ for all } t \geq t_0\right] > 1- \varepsilon.$$
\end{lemma}
\begin{proof}
Since $(\hat \xi_t)$ and $(\doublehat{\xi}_t)$ are constructed from the same Harris system $H$, it suffices to find $t_0$ such that 
$$\mathbb{P}\left[\hat\xi_{t_0} = \doublehat{\xi}_{t_0}\right] > 1- \varepsilon.$$
For a fixed $t_0 > 0$, consider the system of first ancestor processes $((\eta^x_t)_{0 \leq t \leq t_0}: x \in \Z)$ constructed from the time-reversed Harris system $\hat H_{[0,t_0]}$. Since $\hat \xi_0 \equiv \doublehat{\xi}_0$ on $\Z \backslash I$, we have
\begin{align*}
\P\left[\hat \xi_{t_0} = \doublehat{\xi}_{t_0} \right] &\stackrel{\eqref{eq:duality_equation_multi}}{\geq} \P\left[\hat\xi_0(\eta^x_{t_0}) = \doublehat{\xi}_0(\eta^x_{t_0}) \neq 0 \text{ for all } x \in \Z \text{ with } \eta^x_{t_0} \neq \triangle\right] \\&\geq \P\left[\eta^x_{t_0} \notin I \text{ for all } x \in \Z\right] \geq 1-(\#I)\cdot \P\left[0 \in \{\eta^x_{t_0}:x\in\Z\}\right].
\end{align*}
The result now follows from taking $t_0$ large enough, depending on $\varepsilon$ and $L$, by \eqref{eq:density_goes_to_zero}.
\end{proof}

%For $\xi \in \{0,1,2\}^{\Z}$ and $a \in \Z$, define
%\begin{equation}\label{eq:def_Gamma}
%(\Gamma_a \xi) (x)= \mathds{1}_{(\infty, a]}(x) + \xi(x) \cdot \mathds{1}_{(a,\infty)}(x)
%\end{equation}
%and then let $((\Gamma_a \xi)_t)_{t \ge 0}$ be the multitype contact process with initial configuration $%\Gamma_a \xi$. Similarly define 
%\begin{equation}
%(\Psi_a \xi)(x) = \xi(x) \cdot \mathds{1}_{(-\infty, a)}(x) + 2 \cdot \mathds{1}_{[a,\infty)}(x)
%\end{equation}
%and the corresponding process $((\Psi_a \xi)_t)_{t \ge 0}$.

\begin{lemma}
\label{lem:couple_left}
For any $\varepsilon > 0$ there exists $S_0 > 0$ such that the following holds for any $S \ge S_0$. Assume $\xi_0$ satisfies:
\begin{align}
\label{eq:all_1s}&\xi_0 \equiv 2 \text{ on $[0,S]$};\\
\label{eq:no_2s}&r(\xi_0) <  0.
\end{align}
Let $(\xi'_t)_{t\ge 0}$ be the process started from
\begin{equation}\label{eq:def_of_xi_prime} \xi'_0(x) = \mathds{1}_{(-\infty,0)}(x)\cdot \xi_0(x) + 2 \cdot\mathds{1}_{[0,\infty)}(x).
\end{equation}
Then, with probability larger than $1-\varepsilon$ we have
\begin{align}
\label{eq:xi_eq_xip}\text{for any } t \geq 0,\quad &\xi_t \equiv \xi'_t \text{ on } (-\infty, S/2 + \beta t], \\ 
\label{eq:xi_coinc_Gamma}&\ell(\xi_t) = \ell(\xi'_t),\; r(\xi_t) = r(\xi'_t)\text{ and }\\ 
&\label{eq:xi_where_2s}\ell(\xi_t),\;r(\xi_t) < S/2 + \beta t.
\end{align}
\end{lemma}
\begin{proof}
Given $S > 0$, we write
\begin{align*}&f^{(1)}_t = \frac{S}{4} + \frac{\beta}{2} t,\quad f^{(2)}_t = \frac{S}{2} + \beta t,\quad t \geq 0.\end{align*}
Fix $\varepsilon > 0$. By Lemmas \ref{lem:couple_ones}, \ref{lem:desc_bar_sides} and \ref{lem:no_faster}, if $S$ is large enough, then with probability larger than $1-\varepsilon$ all the following three events occur: 
\begin{align*}
&E_1 = \left\{\{\xi_t = 0\} \cap (-\infty,f^{(2)}_t] = \{\xi'_t = 0\} \cap (-\infty, f^{(2)}_t] \text{ for all } t \geq 0\right\},\\[.2cm]
&E_2 = \left\{\text{there exists } x \in (f^{(1)}_t,\;f^{(2)}_t):\;\xi'_t(x) \neq 0 \text{ for all } t \geq 0 \right\},\\[.2cm]
&E_3 = \left\{r(\xi'_t) < f^{(1)}_t \;\text{for all } t \geq 0 \right\}.
\end{align*}
We will also assume that $S > 4R$.

We will now state and prove two auxiliary claims.\\

\noindent \textit{Claim 1.} On $E_1 \cap E_3$, $\{\xi_t = 1\} = \{\xi'_t = 1\},\; \{\xi_t = 2\} \subseteq \{\xi'_t = 2\}$ for all $t$.\\
To see that this holds, first note that
$$\{\xi_0 = 1\} = \{\xi'_0 = 1\},\;\{\xi_0 = 2\} \subseteq \{\xi'_0 = 2\},$$ 
so applying \eqref{eq:attract_multi} we get 
$$\{\xi_t = 1\} \supseteq \{\xi'_t = 1\},\; \{\xi_t = 2\} \subseteq \{\xi'_t = 2\}\text{ for all }t.$$ 
We now fix $(x,t)$ with $\xi_t(x) = 1$ and will show that 
\begin{equation}\xi'_t(x) = 1.\label{eq:tocompcl}\end{equation} Using \eqref{eq:best_path}, it follows from $\xi_t(x) = 1$ that there exists an infection path $\upgamma: [0,t]\to\Z$ so that  $\upgamma(t) = x$,
\begin{align} \label{eq:initial_xi} & \qquad\qquad\quad\;\;\xi_0(\upgamma(0)) = 1,\\
&\xi_{s-}(\upgamma(s)) = 0\text{ whenever }\upgamma(s-) \neq \upgamma(s).\label{eq:nice_jumps_xi} \end{align} 
\eqref{eq:no_2s}, the definition of $\xi'$ and \eqref{eq:initial_xi} imply that
\begin{equation}\label{eq:initial_xip}
\xi'(\upgamma(0)) =1.
\end{equation}
Additionally, the definition of $E_1$ and \eqref{eq:nice_jumps_xi} together give
\begin{equation}
\xi'_{s-}(\upgamma(s)) = 0\text{ for every $s$ such that } \upgamma(s)\leq f^{(2)}_s \text{ and } \upgamma(s-) \neq \upgamma(s). \label{eq:nice_jumps_xip}
\end{equation}
We must have \begin{equation}\label{eq:we_must_have}\upgamma(s) \leq f^{(1)}_s\text{ for every }s \in [0,t],\end{equation}
otherwise we would obtain a contradiction as follows. Let $\bar s$ be the smallest time for which $\upgamma(\bar s)>f^{(1)}_{\bar s}$. Since $f^{(2)}_{\bar s} - f^{(1)}_{\bar s} \geq S/4 > R$ and each jump of $\upgamma$ has size at most $R$, we would have $f^{(1)}_{\bar s} \leq \upgamma(\bar s) \leq f^{(2)}_{\bar s}$. Again using \eqref{eq:best_path}, \eqref{eq:initial_xip} and \eqref{eq:nice_jumps_xip}, we would get $\xi'_{\bar s}(\upgamma(\bar s)) = 1$, so $r(\xi'_{\bar s}) > f^{(1)}_{\bar s}$, contradicting the assumption that $E_3$ occurs. Now, \eqref{eq:nice_jumps_xip}, \eqref{eq:we_must_have} and another application of \eqref{eq:best_path} give \eqref{eq:tocompcl}, so the proof of Claim 1 is complete.\\

\noindent \textit{Claim 2.} On $E_2 \cap E_3$, $\ell(\xi'_t) \leq f^{(2)}_t$ for all $t$.\\
Indeed, by the definition of $E_2$, for any $t$  there exists $x \in [f^{(1)}_t,f^{(2)}_t]$ such that $\xi'_t(x) \neq 0$, and by the definition of $E_3$ we have $r(\xi'_t) < f^{(1)}_t$, so that $x > r(\xi'_t)$, thus $\xi'_t(x) = 2$, thus $\ell(\xi'_t) \leq x \leq f^{(2)}_t$ .\\

We are now ready to conclude. From Claim 1 and the definition of $E_1$, we have that
$$\text{on } E_1 \cap E_3,\;\xi_t(x) = \xi'_t(x) \text{ for all } t \geq 0,\; x \leq f^{(2)}_t.$$
From Claim 1 and the definition of $E_3$,
$$\text{on } E_1 \cap  E_3,\; r(\xi_t) = r(\xi'_t) < f^{(1)}_t < f^{(2)}_t.$$
From Claim 1 and Claim 2,
\begin{equation*}\text{on } E_1 \cap E_2 \cap E_3,\; \ell(\xi_t) = \ell(\xi'_t) < f^{(2)}_t.\end{equation*}
\end{proof}

\begin{corollary}\label{cor:all_tog}
For any $\varepsilon > 0$ there exists $S_0 > 0$ such that the following holds. Assume $S \ge S_0$ and $\xi_0$ satisfies, for some $a,b \in \Z$ with $a< 0 < b$:
\begin{align}
\label{eq_cor_as1}&\xi_0 \equiv 1\text{ on } [a-S, a];\\
\label{eq_cor_as2}&a< r(\xi_0),\;\ell(\xi_0) < b;\\
\label{eq_cor_as3}&\xi_0(x) \equiv 2 \text{ on } [b,b+S].
\end{align}
Let $(\hat \xi_t)_{t \geq 0}$ be the process started from
$$\hat\xi_0(x) = \mathds{1}_{(-\infty, a]}(x) + \mathds{1}_{(a, b)}(x) \cdot \xi_0(x) + 2\cdot\mathds{1}_{[b,\infty)}(x).$$
Then, with probability larger than $1-\varepsilon$, 
\begin{align}\label{eq:couple_int_fin}
\text{ for every } t\ge 0,\quad &r(\xi_t) = r(\hat\xi_t),\quad \ell(\xi_t) = \ell(\hat\xi_t) \text{ and }\\[.2cm]
&a -\frac{S}{2} - \beta t < r(\xi_t),\ell(\xi_t) < b + \frac{S}{2} + \beta t.
\end{align}
\end{corollary}
\begin{proof}
We will also need $(\xi'_t)_{t \geq 0}$, the process started from
$$\xi'_0(x) = \mathds{1}_{(-\infty,b)}(x)\cdot \xi_0(x) + 2 \cdot \mathds{1}_{[b,\infty)}(x).$$
Given $\varepsilon > 0$, by Lemma \ref{lem:couple_left}, $S_0$ can be chosen so that, if \eqref{eq_cor_as2} and \eqref{eq_cor_as3} hold, then  
\begin{align}
\label{eq:eff_coup_1}\mathbb{P}\left[\begin{array}{ll}\text{for every } t \geq 0,&\ell(\xi_t) = \ell(\xi'_t),\; r(\xi_t) = r(\xi'_t)\\[.2cm]&\text{and }\ell(\xi_t), r(\xi_t) < b + \frac{S}{2} + \beta t \end{array}\right]>1-\varepsilon/2.
\end{align}
Now, note that \eqref{eq_cor_as3} and the definition of $\xi'_0$ imply
\begin{align*}
&\xi'_0 \equiv 1 \text{ on } [a-S, a],\\
&r(\xi'_0) > a,
\end{align*}
so that we can again use Lemma \ref{lem:couple_left} (and symmetry) to obtain that
\begin{align}
\mathbb{P}\left[\begin{array}{ll}\text{for any } t \geq 0,&\ell(\xi'_t) = \ell(\hat\xi_t),\;\; r(\xi'_t) = r(\hat \xi_t)\\[.2cm] &\text{and } \ell(\xi'_t), r(\xi'_t) > a - \frac{S}{2} - \beta t\end{array}\right] > 1-\varepsilon/2.\label{eq:eff_coup_inf}
\end{align}
Putting \eqref{eq:eff_coup_1} and \eqref{eq:eff_coup_inf} together, we obtain the desired result.
\end{proof}

\begin{proofof}{\textit{Proposition } \ref{prop:omega_eps_K}.} 
Given $\varepsilon > 0$, we choose $S$ large enough corresponding to $\varepsilon/3$ in Corollary \ref{cor:all_tog}. Increasing $S$ if necessary, by Lemma \ref{lem:no_faster}, we can also assume the following (recall that $r_t = r(\xi^h_t)$ and $\ell_t = \ell(\xi^h_t)$, where $(\xi^h_t)$ is the process started from the heaviside configuration). 
$$\P\left[r_t,\;\ell_t \in [S - \bar \beta t,\; S + \bar \beta t] \text{ for all } t \geq 0\right] > 1 - \varepsilon.$$
Then, given $L > 0$, we choose $t_0$ corresponding to $\varepsilon/3$ and $L$ in Lemma \ref{prop:almost_all_couples_all}.

Now assume $\xi_0 \in \Gamma_{S,L}$. Then, there exist $a < b$ as prescribed in \eqref{eq:def_gamma}; note in particular that $r(\xi_0), \ell(\xi_0) \in (a,b)$, so that $i(\xi_0) \in (a,b)$. Let
\begin{align*}
&\hat \xi_0(x) = \mathds{1}_{(-\infty, a]}(x) + \mathds{1}_{(a, b)}(x) \cdot \xi_0(x) + 2 \cdot \mathds{1}_{[b,\infty)}(x),\\[.2cm]
&\tilde \xi_0(x) = \mathds{1}_{(-\infty, \lfloor i(\xi_0) \rfloor]} + 2\cdot \mathds{1}_{(\lfloor i(\xi_0)\rfloor,\infty)}
\end{align*}
and $(\hat \xi_t)$, $(\tilde \xi_t)$ be the processes started from these configurations. By our choice of $S$ and $t_0$, with probability larger than $1-\varepsilon$ the following three events occur:
\begin{align*}
&\text{for all } t \ge 0,\;i(\xi_t) = i(\hat \xi_t) \in [a - S - \bar \beta t,\;b+S+\bar \beta t];\\[.2cm]
&\text{for all } t \ge 0,\;i(\tilde \xi_t) \in [\lfloor i(\xi_0) \rfloor - S - \bar\beta t,\;\lfloor i(\xi_0) \rfloor + S + \bar \beta t] \subset [a - S - \bar \beta t,\; b+S + \bar \beta t];\\[.2cm]
&\text{for all } t \ge t_0,\;i(\hat \xi_t) = i(\tilde \xi_t).
\end{align*}
If these events all occur, we have 
$$\begin{aligned}&|i(\xi_t) - i(\tilde \xi_t)| \leq b-a+2S +2\bar \beta t_0 \leq L + 2S + 2\bar\beta t_0 \text{ if }t \leq t_0 \text{ and }\\[.2cm] &|i(\xi_t) - i(\tilde \xi_t)|=0 \text{ if } t > t_0.\end{aligned}$$ The desired result now holds for $K = L+2S +2\bar \beta t_0$.
\end{proofof}

\subsection{Proof of Proposition \ref{prop:gamma_sl}}
Recall the definition of $\Omega$ in \eqref{eq:def_of_Omega}. Given $\xi \in \Omega$, we write
$$m(\xi) = \min(r(\xi),\ell(\xi)),\quad M(\xi) = \max(r(\xi),\ell(\xi))$$
and also
$$m_t = m(\xi^h_t) = \min(r_t, \ell_t),\quad M_t = M(\xi^h_t) = \max(r_t,\ell_t).$$
For $t \geq 0$, define
$$X^{(0)}_t = m_t,\qquad X^{(j+1)}_t = \min\{x > X^{(j)}_t: \xi_t(x) \neq 0\},\; j = 0 ,1,\ldots.$$

\begin{lemma}
\label{lem:order_many}
For any $\varepsilon > 0$ and $k \in \N$ there exists $L > 0$ such that
\begin{equation*}\P\left[X^{(k)}_t - X^{(0)}_t > L\right] < \varepsilon \text{ for all } t \geq 0.\end{equation*}
\end{lemma}
\begin{proof}
If the statement is false, then one can find $\delta > 0$, $k \in \N$ and a sequence of times $(t_n)_{n \in \N}$ such that
\begin{equation}\P\left[X^{(k)}_{t_n} - m_{t_n} > n\right] > \delta \text{ for all } n \in \N.
\nonumber\end{equation}
By tightness of the size of the interface (as given by \eqref{eq:interface_tightness}), we can then find $L_0 > 0$ such that
\begin{equation}\P\left[M_{t_n} - m_{t_n} \leq L_0,\; X^{(k)}_{t_n} - m_{t_n} > n\right] > \delta/2\text{ for all } n \in \N.
\label{eq:negation_more}\end{equation}
Let us denote by $E_n$ the event inside the above probability. Note that
\begin{equation}\label{eq:pre_key}\text{if } n > L_0, \text{ then } E_n \subset \{[m_{t_n},M_{t_n}] \subset [m_{t_n},X^{(k)}_{t_n}]\}.\end{equation}

For $n > L_0$, define the event
$$F_n = E_n \cap \left\{\begin{array}{c}\text{in the time interval } [t_n,t_{n+1}], \text{ all vertices in } \\[0.2cm] [m_{t_n} - R,\;m_{t_n}) \cup \{X^{(0)}_{t_n},\ldots, X^{(k)}_{t_n}\} \cup (X^{(k)}_{t_n},\;X^{(k)}_{t_n} + R] \\[.2cm] \text{have a death mark and do not originate any arrow.} \end{array} \right\}$$
Since the set of vertices that appears in the definition of $F_n$ contains $2R + k$ vertices, we have
$
\P[F_n\;|\;E_n] = e^{(-1-2R\lambda)(2R + k)},
$ so that, by \eqref{eq:negation_more},
\begin{equation}\label{eq:pre_key1}
\P[F_n] \geq \frac{\delta}{2}\cdot e^{(-1-2R\lambda)(2R + k)}.\end{equation}
Additionally, by \eqref{eq:pre_key},
\begin{equation} \label{eq:pre_key2}F_n \subseteq \{r_{t_n+1} < m_{t_n},\; \ell_{t_n + 1} > m_{t_n} + n\} \subseteq \{\ell_{t_n + 1} - r_{t_n + 1} > n\}.\end{equation}
Now, \eqref{eq:pre_key1} and \eqref{eq:pre_key2} together imply
$$\text{for all } n > L_0,\; \P\left[\ell_{t_n + 1} - r_{t_n + 1} > n \right] > \frac{\delta}{2}\cdot e^{(-1-2R\lambda)(2R + k)}. $$
This contradicts tightness of the interface size, \eqref{eq:interface_tightness}.
\end{proof}

We will need one extra subset of $\{0,1,2\}^\Z$, defined for $k \in \N$ and $L > k$ by
\begin{equation}\label{eq:def_Pi}\Pi_{k,L} = \left\{\xi \in \Omega:\begin{array}{l}\#\{x \in [m(\xi) - L,\; m(\xi)):\xi(x) = 1\} \geq k,\\[.2cm] \#\{x \in (M(\xi),\; M(\xi) + L]: \xi(x) = 2\} \geq k ,\\[.2cm]M(\xi) - m(\xi) \leq L\end{array}\right\}. \end{equation}

\begin{lemma}\label{lem:gamma_pi_good}
For any $\varepsilon > 0$ and $k \in \N$ there exists $L > 0$ such that, for any $t \geq 0$,
\begin{equation*}\mathbb{P}\left[\xi^h_t \in \Pi_{k,L}\right] > 1- \varepsilon.\end{equation*}
\end{lemma}
\begin{proof}
Fix $\varepsilon > 0$ and $k \in \N$. By \eqref{eq:interface_tightness}, we can choose $L_0$ so that
\begin{equation} \label{eq:again_tight}
\P[M_t - m_t > L_0] < \varepsilon/2 \text{ for all } t \geq 0.\end{equation}
Let $k'= L_0 + k$. We now choose $L > L_0$ corresponding to $\varepsilon/4$ and $k'$ in Lemma \ref{lem:order_many}; we get
\begin{equation}\begin{split}
&\P\left[M_t - m_t \leq L_0,\; \#\{x\in (M_t, M_t + L]: \xi_t(x) = 2\} < k\right] \\[.2cm]&\leq \P\left[\#\{x \in (m_t, M_t + L]: \xi_t(x) \neq 0 \} < k' \right] \\[.2cm]&\leq \P\left[X^{(k')}_t > M_t + L\right] \leq \P\left[X^{(k')}_t > L\right]<  \varepsilon/4.
\end{split} \label{eq:LLprime}\end{equation}
By symmetry we also get
\begin{equation}\label{eq:LLprime2}
\P\left[M_t - m_t \leq L_0,\;\#\{x \in [m_t - L, m_t): \xi_t(x) = 1\} < k\right] < \varepsilon/4.
\end{equation}
The desired statement now follows from putting together \eqref{eq:again_tight} (with the observation that $L_0 < L$), \eqref{eq:LLprime} and \eqref{eq:LLprime2}.
\end{proof}

\begin{lemma}\label{lem:after_one_second}For any $\varepsilon > 0$ and $S > 0$ there exists $k \in \N$ such that, for any $L \geq k$, if $\xi_0 \in \Pi_{k, L}$, then
$$\P\left[\xi_1 \in \Gamma_{S,L} \right] > 1-\varepsilon.$$
\end{lemma}
\begin{proof}
Given $S > 0$ and $x \in \Z$, define the event
$$F_S(x) = \left\{\begin{array}{c}(x,0) \leftrightarrow (y, 1) \text{ for all } y \text{ with } |x-y| \leq S;\\[.2cm] D^{z,y}_{[0,1]} = \varnothing \text{ for all } y,z \text{ with } |x-y|\leq S,\; |x-z| > S \end{array} \right\}.$$
By prescribing the position of a finite number of arrows and the absence of recovery marks and arrows at certain positions, it is easy to show that there exists $\delta_S$ such that, for any $x$,
\begin{equation}\label{eq:fcan_happen}\P[F_S(x)] > \delta_S. \end{equation}
We note that
\begin{align}\label{eq:f_good_prop1}
&F_S(x) \cap \{\xi_0(x) = 1,\;\ell(\xi_0) > x+S\} \subset \{\xi_{1} \equiv 1 \text{ on } [x,x+S],\; \ell(\xi_{1}) > x+S\};\\[.2cm]
\label{eq:f_good_prop2}&F_S(x) \cap \{\xi_0(x) = 1,\;r(\xi_0) < x+S\} \subset \{\xi_{1} \equiv 2 \text{ on } [x-S,x],\; r(\xi_{1}) < x+S\}
\end{align}
and that
\begin{equation}
\label{eq:f_is_indep}
\text{if }x_1, x_2,\ldots \text{ are such that } |x_i - x_j| > 2(R+S) \text{ for all } i \neq j, \text{ then } (F_S(x_i))_{i \geq 1} \text{ are independent}.\end{equation} 

Now, given $\varepsilon > 0$ and $S > 0$, we choose $k \in \N$ so that \begin{equation}\label{eq:choice_of_k_after_delta} k' := \lfloor k/(2R+2S) \rfloor \text{ satisfies } (1-\delta_S)^{k'} < \varepsilon/2\end{equation} (note in particular that $k \geq 2R + 2S$). Assume that $L \geq k$ and $\xi_0 \in \Pi_{k,L}$. By the definition of $\Pi_{k,L}$ in \eqref{eq:def_Pi}, we can find
\begin{align}\label{eq:we_can_find_x} x_1,\ldots, x_{k'} \in [m(\xi_0) - L, m(\xi_0) - 2(S+R)] \text{ with } x_{i+1} > x_i + 2(S+R) \text{ for all }i;\\[.2cm]
y_1,\ldots, y_{k'} \in [M(\xi_0) + 2(S+R), M(\xi_0) +L] \text{ with } y_{i+1} > y_i + 2(S+R) \text{ for all }i\label{eq:we_can_find_y}\end{align}
We then have
\begin{align*}\P\left[\xi_1 \in \Gamma_{S,L}\right] &\stackrel{\eqref{eq:def_gamma},\eqref{eq:f_good_prop1},\eqref{eq:f_good_prop2}}{\geq} \P\left[\left(\cup_{i=1}^k F_S(x_i)\right) \cap \left(\cup_{i=1}^k F_S(y_i)\right)\right] \\[.2cm]&\stackrel{\eqref{eq:fcan_happen},\eqref{eq:f_is_indep},\eqref{eq:we_can_find_x},\eqref{eq:we_can_find_y}}{\geq} 1-2(1-\delta_S)^k > 1 -\varepsilon.\end{align*}
\end{proof}

\begin{proofof}{\textit Proposition \ref{prop:gamma_sl}.}
Fix $\varepsilon$ and $S$. We first choose $k$ corresponding to $\varepsilon/2$ and $S$  in Lemma \ref{lem:after_one_second}, and then choose $L$ corresponding to $\varepsilon/2$ and $k$  in Lemma \ref{lem:gamma_pi_good}. We then have, for any $t \geq 1$,
$$\P\left[\xi^h_t \in \Gamma_{S,L}\right] \geq \P\left[\xi^h_t \in \Gamma_{S,L}\;|\;\xi^h_{t-1} \in \Pi_{k,L}\right] \cdot \P\left[\xi^h_{t-1} \in \Pi_{k,L} \right]\geq (1-\varepsilon/2)^2 > 1-\varepsilon.$$
It is then easy to show that we can increase $L$ if necessary so that the result also holds for $t \in [0,1)$.
\end{proofof}

%%%%%%%%%%%% References %%%%%%%%%%%%%%%%%%%%%%%%%%%%%%%%%%%%%%%%%%%%%%%%%%%%%%%%%%%%%%%%%%%%%%%%%%%%%%%%%%%%%%%%%%%%%%%

\end{document}